\documentclass[11pt,reqno]{amsart}

\usepackage[T1]{fontenc}
\IfFileExists{lmodern.sty}{\usepackage{lmodern}}{}
\usepackage[a4paper,left=25mm,right=25mm,top=26mm,bottom=28mm]{geometry}
\usepackage{amsmath,amssymb,mathrsfs,esint}
\IfFileExists{microtype.sty}{\usepackage{microtype}}{}
\IfFileExists{needspace.sty}{\usepackage{needspace}}{%
  \newcommand{\Needspace}[1]{\par}%
}

\IfFileExists{hyperref.sty}{%
  \usepackage[hidelinks]{hyperref}
}{%
  \usepackage{url}
  \newcommand{\href}[2]{#2}
  \newcommand{\texorpdfstring}[2]{#1}
}

\allowdisplaybreaks[2]
\numberwithin{equation}{section}

\makeatletter
\renewcommand\section{\@startsection{section}{1}{\z@}%
  {4.2ex \@plus 1.2ex \@minus .4ex}%
  {2.0ex \@plus .3ex}%
  {\normalfont\large\bfseries\centering}}
\renewcommand\subsection{\@startsection{subsection}{2}{\z@}%
  {2.2ex \@plus .8ex \@minus .3ex}%
  {.8ex \@plus .2ex}%
  {\normalfont\normalsize\bfseries}}
\renewcommand\subsubsection{\@startsection{subsubsection}{3}{\z@}%
  {1.8ex \@plus .6ex \@minus .2ex}%
  {.6ex \@plus .2ex}%
  {\normalfont\normalsize\bfseries}}
\makeatother

\newtheoremstyle{strongplain}%
  {10pt}
  {10pt}
  {\itshape}
  {0pt}
  {\bfseries}
  {.}
  {0.65em}
  {}

\newtheoremstyle{strongremark}%
  {8pt}
  {8pt}
  {\normalfont}
  {0pt}
  {\bfseries}
  {.}
  {0.65em}
  {}

\theoremstyle{strongplain}
\newtheorem{theorem}{Theorem}[section]
\newtheorem{lemma}[theorem]{Lemma}
\newtheorem{proposition}[theorem]{Proposition}
\newtheorem{corollary}[theorem]{Corollary}
\theoremstyle{strongremark}

\newtheorem{example}[theorem]{Example}

\makeatletter
\renewenvironment{proof}[1][\proofname]{%
  \par\addvspace{4pt}%
  \pushQED{\qed}%
  \normalfont
  \topsep7\p@\@plus3\p@\@minus2\p@
  \trivlist
  \item[\hskip\labelsep{\bfseries\boldmath #1\@addpunct{.}\unboldmath}]%
  \ignorespaces
}{%
  \popQED\endtrivlist\@endpefalse
  \addvspace{4pt}%
}
\makeatother

\newcommand{\proofstep}[1]{%
  \par\addvspace{\medskipamount}%
  \Needspace{5\baselineskip}%
  \noindent{\bfseries\boldmath #1}\par\nobreak\smallskip
}

\newcommand{\Ric}{\operatorname{Ric}}
\newcommand{\diver}{\operatorname{div}}
\newcommand{\supp}{\operatorname{supp}}

\newcommand{\loc}{\mathrm{loc}}
\newcommand{\dd}{\,\mathrm d}
\newcommand{\dmu}{\,\mathrm d\mu}

\newcommand{\A}{\mathcal A}

\title[Liouville theorems for exponential gradient equations]{Liouville theorems for $-\Delta_{m,g}u=\lambda e^u|\nabla u|^q$\\ in the full range of gradient exponents}
\makeatletter
\let\SWZ@setauthors\@setauthors
\def\@setauthors{%
  \begingroup
  \def\andify{\nxandlist{\unskip\qquad}{\unskip\qquad}{\unskip\qquad}}%
  \let\author@andify\andify
  \SWZ@setauthors
  \endgroup
}
\makeatother
\author{Yuhua Sun}
\address{School of Mathematical Sciences and LPMC, Nankai University,
Tianjin 300071, P. R. China}
\email{sunyuhua@nankai.edu.cn}
\author{Zhen Wang}
\address{School of Mathematical Sciences and LPMC, Nankai University,
Tianjin 300071, P. R. China}
\email{lrjwz89757@163.com}
\author{Mengyu Zhang}
\address{School of Mathematical Sciences and LPMC, Nankai University,
Tianjin 300071, P. R. China}
\email{zzmy111@126.com}
\date{}
\begin{document}
\pagestyle{plain}
\begin{abstract}
We study the equation
\[
-\Delta_{m,g}u=\lambda e^u|\nabla u|^q
\]
on complete, connected, noncompact Riemannian manifolds of dimension $n\ge2$ with $\Ric_g\ge0$. For every $m>1$, $q\in\mathbb R$ and $\lambda>0$, we prove that every positive $C^1$ weak solution is constant when $q>0$, while no such solution exists when $q\le0$. For $q>m$, we derive a local gradient estimate independent of the solution by Nash--Moser iteration, which leads to automatic boundedness and resolves the question raised by Lu and Zhu~\cite[Remark~1.6(B)]{LZ}. For $m-1<q<m$, we use an auxiliary function and a divergence inequality to extend the Liouville result beyond the range covered by existing gradient-inequality methods. The remaining cases are treated by suitable changes of variables and integral estimates. The conclusions also extend to real-valued solutions bounded below.
\end{abstract}
\maketitle

\section{Introduction and main results}\label{sec:intro}

Throughout the paper, unless stated otherwise, $(M^n,g)$ is a
connected, complete, noncompact Riemannian manifold with $n\ge2$ and
$\Ric_g\ge0$. We study positive solutions of
\begin{equation}\label{eq:main}
 -\Delta_{m,g}u=\lambda e^u|\nabla u|^q,
 \qquad m>1,\quad q\in\mathbb R,\quad \lambda>0,
\end{equation}
where $\Delta_{m,g}u=\diver_g(|\nabla u|^{m-2}\nabla u)$.
For $q=0$, the right-hand side is $\lambda e^u$ also at critical points;
for $q<0$, the singular gradient factor is understood in the locally
integrable sense specified below. Our aim is to determine the Liouville
property of~\eqref{eq:main} for the entire range of $q$, without assuming
an upper bound on $u$.

The starting point is the pure-gradient equation
$-\Delta u=|\nabla u|^q$. Lions~\cite{Lions85} established constancy of
entire solutions for $q>1$, and Bidaut-V\'eron, Garc\'ia-Huidobro and
V\'eron~\cite{BGV14} extended the theory to quasilinear equations in
the range $q>m-1$. On manifolds, gradient estimates and Liouville
theorems under a Ricci lower bound were obtained by Wang and
Zhang~\cite{WZ11} for $m$-harmonic functions and by Han, He and
Wang~\cite{HanHeWang25} for Hamilton--Jacobi equations.

A closely related model is the power product equation
\[
 -\Delta_m u=u^r|\nabla u|^q\qquad\text{in }\mathbb R^n.
\]
For $m=2$, Bidaut-V\'eron, Garc\'ia-Huidobro and V\'eron~\cite{BGV19}
obtained Liouville regions and the exact threshold for nonconstant
positive radial entire solutions. Further results were established by
Filippucci, Pucci and Souplet~\cite{FPS20} for bounded solutions with
$q>2$, and by Bidaut-V\'eron~\cite{BV21}, who removed boundedness in
the range $1<m<n$, $r\ge0$ and $q\ge m$. Chang, Hu and
Zhang~\cite{CHZ22} treated additional regions with $q<m$.
For $m=2$, Ma and Wu~\cite{MaWu26} enlarged the Liouville region near
the radial threshold, while Lu~\cite{Lu26} obtained improved Liouville
theorems and local estimates, also considering the manifold setting.

Recent progress has substantially clarified the power model.
Lu and Zhu~\cite{LZ} established Liouville results for a broad class
of product nonlinearities when $q>m$.
For the semilinear case, Liang and Zhang~\cite{LiangZhang26}
improved the Liouville ranges and universal estimates.
Most recently, Wu, Yan and Zhu~\cite{WYZ26} showed that, for
$1<m<n$, every positive entire solution of the power equation is
constant for all $r\in\mathbb R$ when $q\ge m-1$; for
$0<q<m-1$ they also covered the region below the second critical
curve. In a companion paper~\cite{WYZ26II}, they classified positive
entire solutions on that critical curve, without imposing a global
growth or energy assumption. On manifolds with $\Ric_g\ge0$,
He, Hu and Wang~\cite{HeHuWang26} used Nash--Moser iteration for
the power model, while Dou, Shi, Wu and Zhu~\cite{DSWZ25}
obtained Liouville and classification results for certain general
reaction functions. The latter hypotheses, however, do not directly
cover $f(u)=e^u$ on the whole positive half-line.

For the exponential product equation~\eqref{eq:main}, the known
Liouville ranges are more limited. If $m\ge n$, the volume growth
under $\Ric_g\ge0$ implies
$m$-parabolicity, so positive $m$-superharmonic functions are constant
\cite{SX21}. Hence the classification is already settled for every
$q$ in this case. The essential range is $1<m<n$. Set
\[
 q_*:=\frac{n(m-1)}{n-1},\qquad m-1<q_*<m.
\]
Since $u>0$, equation~\eqref{eq:main} implies
\[
 -\Delta_{m,g}u\ge\lambda|\nabla u|^q.
\]
Integral test-function and volume-growth methods for differential
inequalities~\cite{MP99,Fil09,Sun15,WX16,SX21,SXX22}, together with
Goffi's area-growth criterion~\cite{Goffi23}, therefore give, on
manifolds with $\Ric_g\ge0$ and without an upper bound on $u$,
\[
 \begin{cases}
 \text{nonexistence of positive solutions},&q\le0,\\
 \text{constancy of positive solutions},&0<q\le q_*.
 \end{cases}
\]
In particular, these conclusions hold in $\mathbb R^n$.

The exponent $q=m$ is also settled: the increasing change of variables
\[
 U=\int_0^u\exp\!\left(\frac{\lambda e^t}{m-1}\right)\dd t
\]
produces a positive $m$-harmonic function, which is constant under
$\Ric_g\ge0$~\cite{WZ11}. For $q>m$, Lu and
Zhu~\cite[Corollary~1.5]{LZ} proved constancy of \emph{bounded}
positive entire solutions in $\mathbb R^n$ and explicitly asked
whether the upper-bound assumption can be removed
\cite[Remark~1.6(B)]{LZ}. Thus, when $1<m<n$, the known results leave
two distinct questions for the exponential product equation:
the interval $q_*<q<m$, not covered even for bounded positive
solutions, and the range $q>m$ without an a priori upper bound.
Neither gap is closed by simply replacing $e^u$ with a power of $u$.
The related additive equation $-\Delta u=e^u+c|\nabla u|^q$
studied in~\cite{BGV26exp} has a different reaction structure.

We resolve both questions on manifolds with $\Ric_g\ge0$.
For $q>m$, Nash--Moser iteration yields a local gradient estimate
independent of $\sup u$; this leads to automatic boundedness and then
constancy. For $m-1<q<m$, an auxiliary function satisfies a uniformly
elliptic divergence inequality, and a global vanishing argument gives
constancy throughout the interval. The remaining cases follow from
the gradient inequality, a change of variables, or direct cutoff
estimates. Together these arguments establish the full classification
for every $q\in\mathbb R$.

If the right-hand side is instead $\lambda|e^u\nabla u|^q
=\lambda e^{qu}|\nabla u|^q$, then for $q>0$ the substitution
$v=qu$ reduces it to~\eqref{eq:main} with a different positive
coefficient. Thus the constancy conclusion also applies to this form;
the nonexistence result for $q\le0$ refers to~\eqref{eq:main}.

\subsection*{Main results}

Write $B(x,R)$ for a geodesic ball and $\dmu=\dd\operatorname{vol}_g$.
We call $u$ a $C^1$ weak solution on $\Omega\subset M$ if
$u\in C^1_{\loc}(\Omega)\cap W^{1,m}_{\loc}(\Omega)$,
$e^u|\nabla u|^q\in L^1_{\loc}(\Omega)$, and
\begin{equation}\label{eq:weak}
 \int_\Omega |\nabla u|^{m-2}\langle\nabla u,\nabla\varphi\rangle\dmu
 =\lambda\int_\Omega e^u|\nabla u|^q\varphi\dmu,
 \qquad \forall \varphi\in C_c^\infty(\Omega).
\end{equation}
For $q<0$, the factor $|\nabla u|^q$ is interpreted as $+\infty$ where $\nabla u=0$.

\begin{theorem}\label{thm:liouville}
Assume $\Ric_g\ge0$. For every $m>1$, $q\in\mathbb R$ and $\lambda>0$,
equation~\eqref{eq:main} admits no nonconstant positive $C^1$
weak solutions on $M$.
\end{theorem}

\begin{corollary}\label{cor:bounded-below}
For $m>1$, $q\in\mathbb R$ and $\lambda>0$, every real-valued
$C^1$ weak solution of~\eqref{eq:main} bounded below is constant.
\end{corollary}

\begin{proof}
If $u\ge L$, setting $v=u-L+1\ge1$ gives
\[
 -\Delta_{m,g}v=(\lambda e^{L-1})e^v|\nabla v|^q.
\]
Theorem~\ref{thm:liouville} applies to $v$.
\end{proof}

\Needspace{12\baselineskip}
\begin{example}\label{ex:lower-bound}
When $m=n=2$ and $q=0$, the smooth function
\[
 u(x)=\log\frac8\lambda-2\log(1+|x|^2),\qquad x\in\mathbb R^2,
\]
satisfies
\[
 -\Delta u=\frac8{(1+|x|^2)^2}=\lambda e^u,
 \qquad \inf_{\mathbb R^2}u=-\infty.
\]
\end{example}

The following local result allows a Ricci lower bound instead of the
standing nonnegative-curvature assumption.

\begin{proposition}\label{thm:gradient}
Let $q>m>1$ and $K\ge0$. Suppose $\Ric_g\ge-(n-1)Kg$ on $B(o,2R)$
and $u>0$ is a $C^1$ weak solution of~\eqref{eq:main} there. Then
\begin{equation}\label{eq:gradient-main}
 \sup_{B(o,R)}|\nabla u|
 \le C(n,m,q)\left[
 \lambda^{-\frac1{q-m}}
 +\lambda^{-\frac1{q-m+1}}(K+R^{-2})^{\frac1{2(q-m+1)}}
 \right].
\end{equation}
The constant $C(n,m,q)$ is independent of the solution, $K$, $R$ and $\lambda$.
\end{proposition}

Section~\ref{sec:prelim} records the analytic preliminaries.
Sections~\ref{sec:local}--\ref{sec:global} treat $q>m$,
Section~\ref{sec:remaining} covers $m-1<q<m$, and
Section~\ref{sec:endpoints} handles the remaining exponents.

\section{Preliminaries}\label{sec:prelim}

The constant $C$ may change from line to line and, unless otherwise stated, depends only on $n,m,q$.
For a ball $B$ of finite positive volume, we write
\[
 \fint_B f\dmu=\frac1{\mu(B)}\int_B f\dmu.
\]
\begin{lemma}[Regularity on the noncritical set]\label{lem:regularity}
If $u\in C^1_{\loc}(\Omega)\cap W^{1,m}_{\loc}(\Omega)$ is a weak solution of~\eqref{eq:main}, then
\[
 u\in C^\infty(\{x\in\Omega:\nabla u(x)\ne0\}).
\]
\end{lemma}

On every relatively compact subdomain of $\{|\nabla u|>0\}$, the right-hand side is bounded and the equation is uniformly elliptic. The local $C^{1,\alpha}$ regularity theory~\cite{DiBenedetto83,Tolksdorf84}, followed by standard Schauder bootstrapping~\cite{GT}, gives the claimed smoothness there.

We use the local Sobolev inequality of Saloff-Coste~\cite[Theorem~3.1]{SC}.

\begin{lemma}\label{lem:sobolev}
Suppose that $\Ric_g\ge -(n-1)K g$ on $B(x,4\rho)$, and set
\[
 \mathcal V=\mu(B(x,2\rho)),\qquad
 \nu=\begin{cases}n,&n>2,\\4,&n=2,\end{cases}
 \qquad \chi=\frac{\nu}{\nu-2}.
\]
Then, for every $f\in C^{\infty}_0(B(x,2\rho))$,
\begin{equation}\label{eq:sobolev}
 \left(\frac1{\mathcal V}\int |f|^{2\chi}\dmu\right)^{1/\chi}
 \le Ce^{C\sqrt K\rho}\rho^2
 \left(
 \frac1{\mathcal V}\int|\nabla f|^2\dmu
 +\frac{\rho^{-2}}{\mathcal V}\int f^2\dmu
 \right),
\end{equation}
where $C$ depends only on $n$. By density, \eqref{eq:sobolev} also holds for $f\in W_0^{1,2}(B(x,2\rho))$.
\end{lemma}

\begin{lemma}\label{lem:harnack}
Assume $\Ric_g\ge0$. For each $m>1$, there exist $\beta=\beta(n,m)>0$ and $C_H=C_H(n,m)>0$ such that every nonnegative continuous weakly $m$-superharmonic function $h$ satisfies
\begin{equation}\label{eq:weak-harnack}
 \left(\fint_{B(x,2R)}h^\beta\dmu\right)^{1/\beta}
 \le C_H\inf_{B(x,R)}h,
 \qquad x\in M,\quad R>0.
\end{equation}
Moreover, for every $t>0$, there exists $C_t=C_t(n,m,t)>0$ such that every nonnegative continuous weakly $m$-subharmonic function $v$ satisfies
\begin{equation}\label{eq:subsolution-bound}
 \sup_{B(x,R)}v
 \le C_t\left(\fint_{B(x,2R)}v^t\dmu\right)^{1/t}.
\end{equation}
These constants are independent of the center, radius and volume of the ball.
\end{lemma}

\begin{proof}
Bishop--Gromov comparison gives
\begin{equation}\label{eq:doubling}
 \mu(B(x,2R))\le2^n\mu(B(x,R)).
\end{equation}
Buser's relative isoperimetric estimate for geodesic balls~\cite[Lemma~5.1]{Buser}, together with the coarea formula and H\"older's inequality, gives
\begin{equation}\label{eq:poincare}
 \fint_{B(x,R)}|f-f_{B(x,R)}|^m\dmu
 \le C(n,m)R^m\fint_{B(x,R)}|\nabla f|^m\dmu,
\end{equation}
where $f_{B(x,R)}=\fint_{B(x,R)}f\dmu$.

By H\"older's inequality,~\eqref{eq:poincare} implies a weak $(1,m)$-Poincar\'e inequality
\begin{equation}\label{eq:weak-poincare}
\fint_{B(x,R)}|f-f_{B(x,R)}|\dmu
 \le C(n,m)R\left(\fint_{B(x,R)}|\nabla f|^m\dmu\right)^{1/m}.
\end{equation}
Thus $M$ supports a weak $(1,m)$-Poincar\'e inequality with unit dilation.

Let $h$ be nonnegative and weakly $m$-superharmonic, and fix an open set $\Omega'\Subset M$.
For every nonnegative $\psi\in C_c^\infty(\Omega')$, convexity gives
\[
\begin{aligned}
 \int_{\Omega'}|\nabla(h+\psi)|^m\dmu
 -\int_{\Omega'}|\nabla h|^m\dmu
 &\ge m\int_{\Omega'}|\nabla h|^{m-2}
       \langle\nabla h,\nabla\psi\rangle\dmu\\
 &\ge0.
\end{aligned}
\]
By density, this implies
\begin{equation}\label{eq:superminimizer}
 \int_{\Omega'}|\nabla h|^m\dmu
 \le\int_{\Omega'}|\nabla(h+\psi)|^m\dmu
\end{equation}
for every nonnegative $\psi\in W_0^{1,m}(\Omega')$.
Thus $h$ is an $m$-superminimizer.
Similarly, nonpositive variations show that every weakly $m$-subharmonic function is an $m$-subminimizer.

The doubling property~\eqref{eq:doubling} and the weak Poincar\'e inequality~\eqref{eq:weak-poincare} therefore allow us to apply
Bj\"orn--Marola~\cite[Theorems~9.2 and~8.2]{BM}, which give~\eqref{eq:weak-harnack} and~\eqref{eq:subsolution-bound}, respectively.
All enlarged balls required by these theorems lie in the domain, since the functions are defined on the whole manifold.
The constants depend only on the doubling and Poincar\'e constants and on the stated exponents, and are therefore independent of the center and radius.
Finally, continuity allows the essential infima and suprema in these theorems to be replaced by the ordinary infima and suprema.
\end{proof}

\begin{lemma}\label{lem:hessian}
Let $u$ be a $C^1$ weak solution of~\eqref{eq:main}, with $m>1$ and
$q\in\mathbb R$. On $\{|\nabla u|>0\}$, we have
\begin{equation}\label{eq:hessian}
 |\nabla^2u|^2+(m-2)\bigl|\nabla|\nabla u|\bigr|^2
 \ge\frac{\lambda^2e^{2u}}{n+m-2}|\nabla u|^{2q-2m+4}.
\end{equation}
Moreover, in an orthonormal frame with
$e_1=\nabla u/|\nabla u|$, writing $u_{ij}=\nabla^2u(e_i,e_j)$, we have
\begin{equation}\label{eq:transverse-hessian}
 \sum_{i,j=2}^n u_{ij}^2
 \ge\frac1{n-1}
 \left((m-1)u_{11}+\lambda e^u|\nabla u|^{q-m+2}\right)^2.
\end{equation}
\end{lemma}

\begin{proof}
By Lemma~\ref{lem:regularity}, $u$ is smooth on $\{|\nabla u|>0\}$.
The equation gives
\[
 |\nabla u|^{m-2}\left((m-1)u_{11}+\sum_{i=2}^n u_{ii}\right)
 =-\lambda e^u|\nabla u|^q.
\]
Dividing by $|\nabla u|^{m-2}$, we obtain
\begin{equation}\label{eq:trace}
 (m-1)u_{11}+\sum_{i=2}^n u_{ii}
 =-\lambda e^u|\nabla u|^{q-m+2}.
\end{equation}
Since $u_1=|\nabla u|$ and $u_j=0$ for $2\le j\le n$, we have
\begin{equation}\label{eq:gradient-frame}
\nabla_i|\nabla u|
=\frac{\nabla^2u(e_i,\nabla u)}{|\nabla u|}
=u_{i1}.
\end{equation}
Consequently,
\[
\bigl|\nabla|\nabla u|\bigr|^2
=u_{11}^2+\sum_{i=2}^n u_{1i}^2.
\]
On the other hand, the symmetry $u_{ij}=u_{ji}$ gives
\[
|\nabla^2u|^2
=\sum_{i,j=1}^n u_{ij}^2
=u_{11}^2+2\sum_{i=2}^n u_{1i}^2
+\sum_{i,j=2}^n u_{ij}^2.
\]
Combining these identities, we obtain
\[
\begin{aligned}
|\nabla^2u|^2+(m-2)\bigl|\nabla|\nabla u|\bigr|^2
=(m-1)u_{11}^2+m\sum_{i=2}^n u_{1i}^2
+\sum_{i,j=2}^n u_{ij}^2.
\end{aligned}
\]
The weighted Cauchy--Schwarz inequality yields
\[
\begin{aligned}
\left((m-1)u_{11}+\sum_{i=2}^n u_{ii}\right)^2
&=\left(
\sqrt{m-1}\,\bigl(\sqrt{m-1}\,u_{11}\bigr)
+\sum_{i=2}^n 1\cdot u_{ii}
\right)^2\\
&\le
\left((m-1)+\sum_{i=2}^n1\right)
\left((m-1)u_{11}^2+\sum_{i=2}^n u_{ii}^2\right)\\
&=(n+m-2)
\left((m-1)u_{11}^2+\sum_{i=2}^n u_{ii}^2\right)\\
 &\le(n+m-2)\left(|\nabla^2u|^2+(m-2)\bigl|\nabla|\nabla u|\bigr|^2\right).
\end{aligned}
\]
Together with~\eqref{eq:trace}, this proves~\eqref{eq:hessian} for every $m>1$.

Finally, the Cauchy--Schwarz inequality in the directions orthogonal to
$\nabla u$, together with~\eqref{eq:trace}, gives
\[
\begin{aligned}
\sum_{i,j=2}^n u_{ij}^2
&\ge \sum_{i=2}^n u_{ii}^2
 \ge \frac1{n-1}\left(\sum_{i=2}^n u_{ii}\right)^2\\
&=\frac1{n-1}
 \left((m-1)u_{11}+\lambda e^u|\nabla u|^{q-m+2}\right)^2.
\end{aligned}
\]
This proves~\eqref{eq:transverse-hessian}.
\end{proof}

\section{Local gradient estimates in the range \texorpdfstring{$q>m$}{q > m}}\label{sec:local}

In this section we assume $q>m$ and prove Proposition~\ref{thm:gradient}.
Fix a positive solution $u$ as in that proposition.
Until Sobolev regularity across the critical set has been established, all integrals involving $\nabla|\nabla u|$ or $\nabla^2u$ are taken over $\{|\nabla u|>0\}$.

\begin{lemma}\label{lem:bochner}
On $\{|\nabla u|>0\}$, we have
\begin{equation}\label{eq:bochner}
 \begin{aligned}
 &\frac12\diver\left[|\nabla u|^{m-2}
   \left(I+(m-2)\frac{\nabla u\otimes\nabla u}{|\nabla u|^2}\right)\nabla(|\nabla u|^2)\right]\\
 &\quad=|\nabla u|^{m-2}\left(|\nabla^2u|^2+(m-2)\bigl|\nabla|\nabla u|\bigr|^2
                  +\Ric(\nabla u,\nabla u)\right)\\
 &\qquad\quad-\lambda e^u|\nabla u|^{q+2}
                  -q\lambda e^u|\nabla u|^{q-1}\langle\nabla u,\nabla|\nabla u|\rangle.
 \end{aligned}
\end{equation}
\end{lemma}

\begin{proof}
By Lemma~\ref{lem:regularity}, $u$ is smooth on $\{|\nabla u|>0\}$. For any $p\in\{|\nabla u|>0\}$, we choose an orthonormal frame, so it is obvious that
\[
 \nabla_k(|\nabla u|^{m-2}u^i)
 =|\nabla u|^{m-2}\left(\delta^{ij}+(m-2)\frac{u^iu^j}{|\nabla u|^2}\right)u_{jk}.
\]
Differentiating once more and commuting the covariant derivatives of
$|\nabla u|^{m-2}\nabla u$ gives
\[
\begin{aligned}
&\frac12\diver\left[
|\nabla u|^{m-2}
\left(I+(m-2)\frac{\nabla u\otimes\nabla u}{|\nabla u|^2}\right)
\nabla(|\nabla u|^2)
\right]\\
&=\nabla_i\left[
u^j\nabla_j\bigl(|\nabla u|^{m-2}u^i\bigr)
\right]\\
&=u_{ij}\nabla_j\bigl(|\nabla u|^{m-2}u^i\bigr)
+u^j\nabla_i\nabla_j\bigl(|\nabla u|^{m-2}u^i\bigr)\\
&=|\nabla u|^{m-2}|\nabla^2u|^2
+(m-2)|\nabla u|^{m-4}u_i u_{ij}u_k u_{kj}
+u^j\nabla_j\left[
\nabla_i\bigl(|\nabla u|^{m-2}u^i\bigr)
\right]
+|\nabla u|^{m-2}\Ric_{jk}u^ju^k\\
&=|\nabla u|^{m-2}
\left(
|\nabla^2u|^2
+(m-2)\bigl|\nabla|\nabla u|\bigr|^2
+\Ric(\nabla u,\nabla u)
\right)+\langle\nabla u,\nabla\Delta_{m,g}u\rangle.
\end{aligned}
\]
Substituting $\Delta_{m,g}u=-\lambda e^u|\nabla u|^q$ proves~\eqref{eq:bochner}.
\end{proof}

\begin{lemma}\label{lem:energy}
For any $r\ge\max\left\{1,2-m,\frac{2(n+m-2)q^2}{m-1}\right\}$
and any nonnegative $\eta\in C_0^\infty(B(o,2R))$, we have
\begin{equation}\label{eq:energy}
 \begin{aligned}
 &\frac r2\int\eta^2|\nabla u|^{m+r-2}
   \left(\bigl|\nabla|\nabla u|\bigr|^2+(m-2)\frac{\langle\nabla u,\nabla|\nabla u|\rangle^2}{|\nabla u|^2}\right)\dmu\\
 &\quad+\frac{\lambda^2}{4(n+m-2)}
       \int\eta^2e^{2u}|\nabla u|^{2q-m+r+2}\dmu\\
 &\le C\left[
   \left(K+\lambda^{-2/(q-m)}\right)\int\eta^2|\nabla u|^{m+r}\dmu
   +\int |\nabla u|^{m+r}|\nabla\eta|^2\dmu\right].
 \end{aligned}
\end{equation}
Moreover, $|\nabla u|^{(m+r)/2}\in W^{1,2}_{\loc}(B(o,2R))$ and
\begin{equation}\label{eq:moser-energy}
 \begin{aligned}
 \int\bigl|\nabla(\eta|\nabla u|^{(m+r)/2})\bigr|^2\dmu
 \le C(m+r)^2\bigg[&
   \left(K+\lambda^{-2/(q-m)}\right)\int\eta^2|\nabla u|^{m+r}\dmu\\
   &+\int |\nabla u|^{m+r}|\nabla\eta|^2\dmu\bigg].
 \end{aligned}
\end{equation}
The constant $C$ is uniform for all such $r$.
\end{lemma}

\begin{proof}
Fix $r$ as in the statement and $\delta>0$.
On $\{|\nabla u|>\delta\}$, multiplying
\eqref{eq:bochner} by
$\eta^2(|\nabla u|^r-\delta^r)$
and using the product rule for divergence, we have
\[
\begin{aligned}
&\eta^2(|\nabla u|^r-\delta^r)
\operatorname{div}\left[
|\nabla u|^{m-1}
\left(
I+(m-2)\frac{\nabla u\otimes\nabla u}{|\nabla u|^2}
\right)\nabla|\nabla u|
\right]\\
&=\operatorname{div}\left[
\eta^2(|\nabla u|^r-\delta^r)|\nabla u|^{m-1}
\left(
I+(m-2)\frac{\nabla u\otimes\nabla u}{|\nabla u|^2}
\right)\nabla|\nabla u|
\right]\\
&\quad-2\eta(|\nabla u|^r-\delta^r)|\nabla u|^{m-1}
\left\langle
\left(
I+(m-2)\frac{\nabla u\otimes\nabla u}{|\nabla u|^2}
\right)\nabla|\nabla u|,\nabla\eta
\right\rangle\\
&\quad-r\eta^2|\nabla u|^{m+r-2}
\left(
|\nabla|\nabla u||^2
+(m-2)
\frac{\langle\nabla u,\nabla|\nabla u|\rangle^2}
{|\nabla u|^2}
\right).
\end{aligned}
\]
Substituting \eqref{eq:bochner},
we obtain
\[
\begin{aligned}
&r\eta^2|\nabla u|^{m+r-2}
\left(
|\nabla|\nabla u||^2
+(m-2)
\frac{\langle\nabla u,\nabla|\nabla u|\rangle^2}
{|\nabla u|^2}
\right)\\
&\quad+\eta^2|\nabla u|^{m-2}
(|\nabla u|^r-\delta^r)
\left(
|\nabla^2u|^2
+(m-2)|\nabla|\nabla u||^2
+\operatorname{Ric}(\nabla u,\nabla u)
\right)\\
&=\operatorname{div}\left[
\eta^2(|\nabla u|^r-\delta^r)|\nabla u|^{m-1}
\left(
I+(m-2)\frac{\nabla u\otimes\nabla u}{|\nabla u|^2}
\right)\nabla|\nabla u|
\right]\\
&\quad-2\eta(|\nabla u|^r-\delta^r)|\nabla u|^{m-1}
\left\langle
\left(
I+(m-2)\frac{\nabla u\otimes\nabla u}{|\nabla u|^2}
\right)\nabla|\nabla u|,\nabla\eta
\right\rangle\\
&\quad+\lambda\eta^2e^u|\nabla u|^{q+2}
(|\nabla u|^r-\delta^r)\\
&\quad+q\lambda\eta^2e^u|\nabla u|^{q-1}
(|\nabla u|^r-\delta^r)
\langle\nabla u,\nabla|\nabla u|\rangle.
\end{aligned}
\]
Integrating over $\{|\nabla u|>\delta\}$, the divergence term vanishes by the compact support of $\eta$ and the
zero value of the truncation on $\{|\nabla u|=\delta\}$.
Using \eqref{eq:hessian} and
$\operatorname{Ric}(\nabla u,\nabla u)
\ge -(n-1)K|\nabla u|^2$,
we obtain
\[
\begin{aligned}
&r\int_{|\nabla u|>\delta}\eta^2|\nabla u|^{m+r-2}
 \left(\bigl|\nabla|\nabla u|\bigr|^2
 +(m-2)\frac{\langle\nabla u,\nabla|\nabla u|\rangle^2}{|\nabla u|^2}\right)\dmu\\
&\quad+\frac{\lambda^2}{n+m-2}\int_{|\nabla u|>\delta}
 \eta^2e^{2u}|\nabla u|^{2q-m+2}(|\nabla u|^r-\delta^r)\dmu\\
&\le \int_{|\nabla u|>\delta}\eta^2\lambda e^u
 |\nabla u|^{q+2}(|\nabla u|^r-\delta^r)\dmu
 +(n-1)K\int_{|\nabla u|>\delta}\eta^2
 |\nabla u|^m(|\nabla u|^r-\delta^r)\dmu\\
&\quad+q\int_{|\nabla u|>\delta}\eta^2\lambda e^u
 |\nabla u|^{q-1}(|\nabla u|^r-\delta^r)
 \bigl|\langle\nabla u,\nabla|\nabla u|\rangle\bigr|\dmu\\
&\quad+2\int_{|\nabla u|>\delta}
 \eta|\nabla u|^{m-1}(|\nabla u|^r-\delta^r)
 \left|\left\langle
 \left(I+(m-2)\frac{\nabla u\otimes\nabla u}{|\nabla u|^2}\right)
 \nabla|\nabla u|,\nabla\eta\right\rangle\right|\dmu.
\end{aligned}
\]
We first estimate the first term on the right-hand side.
By Young's inequality $ab\le
\varepsilon a^{\frac{2(q-m+1)}{q-m+2}}
+C_\varepsilon
b^{\frac{2(q-m+1)}{q-m}}
$, we obtain
\begin{equation}\label{eq:young-source}
\begin{aligned}
&\int_{|\nabla u|>\delta}
\eta^2\lambda e^u|\nabla u|^{q+2}
(|\nabla u|^r-\delta^r)\,d\mu\\
&\le \frac{\lambda^2}{4(n+m-2)}
\int_{|\nabla u|>\delta}
\eta^2e^{2u}|\nabla u|^{2q-m+2}
(|\nabla u|^r-\delta^r)\,d\mu\\
&\quad+C\lambda^{-2/(q-m)}
\int_{|\nabla u|>\delta}
\eta^2|\nabla u|^m
(|\nabla u|^r-\delta^r)\,d\mu.
\end{aligned}
\end{equation}
Next, by the Cauchy--Schwarz inequality, we have
\[
|\nabla|\nabla u||^2
+(m-2)
\frac{\langle\nabla u,\nabla|\nabla u|\rangle^2}
{|\nabla u|^2}\ge
(m-1)
\frac{\langle\nabla u,\nabla|\nabla u|\rangle^2}
{|\nabla u|^2}.
\]
Moreover, on $\{|\nabla u|>\delta\}$,
\begin{equation}\label{eqdelta}
(|\nabla u|^r-\delta^r)^2\le|\nabla u|^r(|\nabla u|^r-\delta^r).
\end{equation}
Applying Young's inequality
$ab\le a^2/4+b^2$, we obtain
\[
\begin{aligned}
&q\lambda e^u|\nabla u|^{q-1}
(|\nabla u|^r-\delta^r)
|\langle\nabla u,\nabla|\nabla u|\rangle|\\
&\le
\frac{r(m-1)}4
|\nabla u|^{m+r-4}
\langle\nabla u,\nabla|\nabla u|\rangle^2+
\frac{q^2}{r(m-1)}
\lambda^2e^{2u}|\nabla u|^{2q-m-r+2}
(|\nabla u|^r-\delta^r)^2\\
&\le
\frac r4|\nabla u|^{m+r-2}
\left(
|\nabla|\nabla u||^2
+(m-2)
\frac{\langle\nabla u,\nabla|\nabla u|\rangle^2}
{|\nabla u|^2}
\right)+
\frac{q^2}{r(m-1)}
\lambda^2e^{2u}|\nabla u|^{2q-m+2}
(|\nabla u|^r-\delta^r).
\end{aligned}
\]
Multiplying by $\eta^2$ and integrating over
$\{|\nabla u|>\delta\}$ yields
\begin{equation}\label{eq:young-drift}
\begin{aligned}
&q\int_{|\nabla u|>\delta}
\eta^2\lambda e^u|\nabla u|^{q-1}
(|\nabla u|^r-\delta^r)
\bigl|\langle\nabla u,\nabla|\nabla u|\rangle\bigr|
\,d\mu\\
&\le\frac r4\int_{|\nabla u|>\delta}
\eta^2|\nabla u|^{m+r-2}
\left(
|\nabla|\nabla u||^2
+(m-2)
\frac{\langle\nabla u,\nabla|\nabla u|\rangle^2}
{|\nabla u|^2}
\right)\,d\mu\\
&\quad+\frac{q^2}{r(m-1)}
\int_{|\nabla u|>\delta}
\eta^2\lambda^2e^{2u}
|\nabla u|^{2q-m+2}
(|\nabla u|^r-\delta^r)\,d\mu.
\end{aligned}
\end{equation}
Finally, we estimate the last term. Since the eigenvalues of $
I+(m-2)\frac{\nabla u\otimes\nabla u}{|\nabla u|^2}
$ are $1$ and $m-1$, the Cauchy--Schwarz inequality gives
\[
\begin{aligned}
&
\left\langle
\left(
I+(m-2)\frac{\nabla u\otimes\nabla u}{|\nabla u|^2}
\right)
\nabla|\nabla u|,\nabla\eta
\right\rangle^2
\le
\max\{1,m-1\}|\nabla\eta|^2
\left(
|\nabla|\nabla u||^2
+(m-2)
\frac{\langle\nabla u,\nabla|\nabla u|\rangle^2}
{|\nabla u|^2}
\right).
\end{aligned}
\]
by \eqref{eqdelta}
and Young's inequality $2ab\le\frac r4a^2+\frac4r b^2,$ we obtain
\[
\begin{aligned}
&2\eta|\nabla u|^{m-1}
(|\nabla u|^r-\delta^r)
\left|
\left\langle
\left(
I+(m-2)\frac{\nabla u\otimes\nabla u}{|\nabla u|^2}
\right)
\nabla|\nabla u|,\nabla\eta
\right\rangle
\right|\\
&\le
2\sqrt{\max\{1,m-1\}}\,
\eta|\nabla u|^{m-1}
(|\nabla u|^r-\delta^r)
|\nabla\eta|\times
\left(
|\nabla|\nabla u||^2
+(m-2)
\frac{\langle\nabla u,\nabla|\nabla u|\rangle^2}
{|\nabla u|^2}
\right)^{1/2}\\
&\le
2\sqrt{\max\{1,m-1\}}\,
\eta|\nabla u|^{m+r/2-1}
(|\nabla u|^r-\delta^r)^{1/2}
|\nabla\eta|\times
\left(
|\nabla|\nabla u||^2
+(m-2)
\frac{\langle\nabla u,\nabla|\nabla u|\rangle^2}
{|\nabla u|^2}
\right)^{1/2}\\
&\le
\frac r4\eta^2|\nabla u|^{m+r-2}
\left(
|\nabla|\nabla u||^2
+(m-2)
\frac{\langle\nabla u,\nabla|\nabla u|\rangle^2}
{|\nabla u|^2}
\right)+
\frac{4\max\{1,m-1\}}r
|\nabla u|^m
(|\nabla u|^r-\delta^r)
|\nabla\eta|^2.
\end{aligned}
\]
Integrating over $\{|\nabla u|>\delta\}$ yields
\begin{equation}\label{eq:young-cutoff}
\begin{aligned}
&2\int_{|\nabla u|>\delta}
\eta|\nabla u|^{m-1}
(|\nabla u|^r-\delta^r)
\left|
\left\langle
\left(
I+(m-2)\frac{\nabla u\otimes\nabla u}{|\nabla u|^2}
\right)
\nabla|\nabla u|,\nabla\eta
\right\rangle
\right|\,d\mu\\
&\le
\frac r4\int_{|\nabla u|>\delta}
\eta^2|\nabla u|^{m+r-2}
\left(
|\nabla|\nabla u||^2
+(m-2)
\frac{\langle\nabla u,\nabla|\nabla u|\rangle^2}
{|\nabla u|^2}
\right)\,d\mu\\
&\quad+
\frac{4\max\{1,m-1\}}r
\int_{|\nabla u|>\delta}
|\nabla u|^m
(|\nabla u|^r-\delta^r)
|\nabla\eta|^2\,d\mu.
\end{aligned}
\end{equation}
By the choice of $r$ in the assumption of lemma, we have $\frac{q^2}{r(m-1)}
\le\frac1{2(n+m-2)}.$ Combining \eqref{eq:young-source},
\eqref{eq:young-drift}, and
\eqref{eq:young-cutoff} with the preceding
integral inequality and absorbing the corresponding
terms into the left-hand side, we obtain
\begin{equation}\label{eq:uniform-cutoff-energy}
\begin{aligned}
&\frac r2\int_{|\nabla u|>\delta}
\eta^2|\nabla u|^{m+r-2}
\left(
|\nabla|\nabla u||^2
+(m-2)
\frac{\langle\nabla u,\nabla|\nabla u|\rangle^2}
{|\nabla u|^2}
\right)\,d\mu\\
&\quad+\frac{\lambda^2}{4(n+m-2)}
\int_{|\nabla u|>\delta}
\eta^2e^{2u}|\nabla u|^{2q-m+2}
(|\nabla u|^r-\delta^r)\,d\mu\\
&\le C\left[
\left(K+\lambda^{-2/(q-m)}\right)
\int_{|\nabla u|>\delta}
\eta^2|\nabla u|^m
(|\nabla u|^r-\delta^r)\,d\mu
\right.\\
&\hspace{22mm}\left.
+\int_{|\nabla u|>\delta}
|\nabla u|^m(|\nabla u|^r-\delta^r)
|\nabla\eta|^2\,d\mu
\right].
\end{aligned}
\end{equation}
Here $C=C(n,m,q)$ is independent of $r$ and $\delta$, since $r\ge1$.
As $\delta\downarrow0$, the two terms on the left of \eqref{eq:uniform-cutoff-energy}
increase, whereas its right-hand side is dominated on $\supp\eta$ by the local boundedness of $|\nabla u|$.
Thus monotone and dominated convergence yield \eqref{eq:energy}.
In particular, by the Cauchy--Schwarz inequality,
\[
 \bigl|\nabla|\nabla u|\bigr|^2
 +(m-2)\frac{\langle\nabla u,\nabla|\nabla u|\rangle^2}{|\nabla u|^2}
 \ge\min\{1,m-1\}\bigl|\nabla|\nabla u|\bigr|^2.
\]
Hence, dropping the second nonnegative term in \eqref{eq:energy}, we obtain
\[
 \int\eta^2|\nabla u|^{m+r-2}\bigl|\nabla|\nabla u|\bigr|^2\dmu
 \le\frac Cr\left[
 \left(K+\lambda^{-2/(q-m)}\right)\int\eta^2|\nabla u|^{m+r}\dmu
 +\int|\nabla u|^{m+r}|\nabla\eta|^2\dmu
 \right].
\]

To justify the subsequent use of the chain rule,
consider the truncations
\[
\left(
|\nabla u|^{(m+r)/2}
-\delta^{(m+r)/2}
\right)_+.
\]
Since
$u$ is smooth on $\{|\nabla u|>0\}$.
Moreover,
\[
\begin{aligned}
&\int\eta^2
\left|
\nabla\left(
|\nabla u|^{(m+r)/2}
-\delta^{(m+r)/2}
\right)_+
\right|^2\dmu\\
&=\frac{(m+r)^2}{4}
\int_{|\nabla u|>\delta}
\eta^2|\nabla u|^{m+r-2}
\bigl|\nabla|\nabla u|\bigr|^2\dmu\\
&\le\frac{C(m+r)^2}{r}\left[
\left(K+\lambda^{-2/(q-m)}\right)
\int\eta^2|\nabla u|^{m+r}\dmu
+\int|\nabla u|^{m+r}|\nabla\eta|^2\dmu
\right].
\end{aligned}
\]
Since the right-hand side is independent of $\delta$, choosing $\eta=1$ on any relatively compact subdomain
gives a uniform local $L^2$ bound for the gradients of these truncations.
Furthermore,
\[
0\le
\left(
|\nabla u|^{(m+r)/2}-\delta^{(m+r)/2}
\right)_+
\le |\nabla u|^{(m+r)/2},
\]
which is locally bounded because $u\in C^1_{\loc}$.
Thus the truncations are uniformly bounded in $W^{1,2}$
on every relatively compact subdomain.

As $\delta\downarrow0$, they converge locally uniformly to $|\nabla u|^{(m+r)/2}$.
By weak compactness in $W^{1,2}$,
\[
|\nabla u|^{(m+r)/2}\in W^{1,2}_{\loc}(B(o,2R)).
\]
Its weak gradient vanishes almost everywhere on $\{|\nabla u|=0\}$.

We can now apply the product and chain rules to obtain
\[
\begin{aligned}
&\left|\nabla\left(\eta|\nabla u|^{(m+r)/2}\right)\right|^2\\
&=\left|
|\nabla u|^{(m+r)/2}\nabla\eta
+\frac{m+r}{2}\eta|\nabla u|^{(m+r-2)/2}\nabla|\nabla u|
\right|^2\\
&\le2|\nabla u|^{m+r}|\nabla\eta|^2
+\frac{(m+r)^2}{2}\eta^2|\nabla u|^{m+r-2}\bigl|\nabla|\nabla u|\bigr|^2.
\end{aligned}
\]
Integrating and using the preceding energy estimate, we obtain
\[
\begin{aligned}
&\int\left|\nabla\left(\eta|\nabla u|^{(m+r)/2}\right)\right|^2\dmu\\
&\le2\int|\nabla u|^{m+r}|\nabla\eta|^2\dmu\\
&\quad+\frac{C(m+r)^2}{r}\left[
\left(K+\lambda^{-2/(q-m)}\right)
\int\eta^2|\nabla u|^{m+r}\dmu
+\int|\nabla u|^{m+r}|\nabla\eta|^2\dmu
\right]\\
&\le C(m+r)^2\left[
\left(K+\lambda^{-2/(q-m)}\right)
\int\eta^2|\nabla u|^{m+r}\dmu
+\int|\nabla u|^{m+r}|\nabla\eta|^2\dmu
\right],
\end{aligned}
\]
where we used $r\ge1$ in the last step.
This proves \eqref{eq:moser-energy}.
The integration by parts above is justified by smooth approximation of the compactly supported truncation on the noncritical set.
\end{proof}

\begin{lemma}\label{lem:initial}
Suppose that $B(x,4\rho)\Subset B(o,2R)$, and fix
\[
r=\max\left\{1,2-m,\frac{2(n+m-2)q^2}{m-1}\right\}.
\]
Then
\begin{equation}\label{eq:initial}
 \begin{aligned}
 &\left(\frac1{\mu(B(x,2\rho))}
     \int_{B(x,3\rho/2)}|\nabla u|^{r+2q-m+2}\dmu\right)^{1/(r+2q-m+2)}\\
 &\qquad\le C\left[\lambda^{-2}
     \left(K+\lambda^{-2/(q-m)}+\rho^{-2}\right)\right]^{1/(2q-2m+2)}.
 \end{aligned}
\end{equation}
\end{lemma}

\begin{proof}
Choose $\xi\in C_0^\infty(B(x,2\rho))$ such that $0\le\xi\le1$,
$\xi=1$ on $B(x,3\rho/2)$, and $|\nabla\xi|\le C/\rho$.
Choose an integer $N>\frac{r+2q-m+2}{2q-2m+2}$ and take $\eta=\xi^N$ in \eqref{eq:energy}.
Discarding the derivative term and using $e^u\ge1$, we obtain
\[
 \begin{aligned}
 &\lambda^2\int_{B(x,2\rho)}
       \xi^{2N}|\nabla u|^{r+2q-m+2}\dmu\\
 &\qquad\le C\left(K+\lambda^{-2/(q-m)}+\rho^{-2}\right)
       \int_{B(x,2\rho)}
       \xi^{2N-2}|\nabla u|^{m+r}\dmu.
 \end{aligned}
\]
Here
\[
 |\nabla\eta|^2\le C\rho^{-2}\xi^{2N-2},
 \qquad
 \eta^2\le\xi^{2N-2}.
\]
Since $2N-2\ge\frac{2N(m+r)}{r+2q-m+2}$ and $0\le\xi\le1$, H\"older's inequality gives
\[
 \begin{aligned}
 &\int_{B(x,2\rho)}
       \xi^{2N-2}|\nabla u|^{m+r}\dmu\\
 &\quad\le
 \left(\int_{B(x,2\rho)}
       \xi^{2N}|\nabla u|^{r+2q-m+2}\dmu
       \right)^{(m+r)/(r+2q-m+2)}\\
 &\qquad\qquad\cdot\mu(B(x,2\rho))^{(2q-2m+2)/(r+2q-m+2)}.
 \end{aligned}
\]
If the integral on the right is zero, the conclusion is immediate.
Otherwise, dividing the preceding energy estimate by its power
$(m+r)/(r+2q-m+2)$ gives
\[
 \begin{aligned}
 &\left(\frac1{\mu(B(x,2\rho))}\int_{B(x,2\rho)}
       \xi^{2N}|\nabla u|^{r+2q-m+2}\dmu
       \right)^{(2q-2m+2)/(r+2q-m+2)}\\
 &\qquad\le C\lambda^{-2}
       \left(K+\lambda^{-2/(q-m)}+\rho^{-2}\right).
 \end{aligned}
\]
Since $\xi=1$ on the inner ball, this proves~\eqref{eq:initial}.
\end{proof}

\begin{proof}[Proof of Proposition~\ref{thm:gradient}]
Fix $x\in B(o,R)$ and choose
\begin{equation}\label{eq:scale}
 \rho=\frac1{16}\min\{R,K^{-1/2},\lambda^{1/(q-m)}\},
\end{equation}
where $K^{-1/2}$ is understood to be $+\infty$ when $K=0$.
Then $B(x,4\rho)\Subset B(o,2R)$, and
\[
 K\rho^2\le\frac1{256},\qquad
 \lambda^{-2/(q-m)}\rho^2\le\frac1{256}.
\]
Hence both the exponential factor in Lemma~\ref{lem:sobolev} and
$(K+\lambda^{-2/(q-m)})\rho^2$ are uniformly bounded.

Keep $r$ as in Lemma~\ref{lem:initial}, and choose the iteration exponents and radii
\[
 \gamma_j=(r+2q-m+2)\chi^j,\qquad
 t_j=\rho(1+2^{-j-1}),\qquad j=0,1,\ldots.
\]
Choose $\eta_j\in C_0^\infty(B(x,t_j))$ such that
$0\le\eta_j\le1$, $\eta_j=1$ on $B(x,t_{j+1})$, and
$|\nabla\eta_j|\le2^{j+3}/\rho$.
At step $j$, apply Lemma~\ref{lem:energy} with the exponent
$r_j=\gamma_j-m\ge r+2(q-m+1)>r$.
Its hypotheses are therefore satisfied, and the function
$\eta_j|\nabla u|^{\gamma_j/2}$ belongs to $W^{1,2}_0(B(x,2\rho))$.
Substituting it into \eqref{eq:sobolev} and applying \eqref{eq:moser-energy}, we obtain
\begin{equation}\label{eq:iteration}
 \begin{aligned}
 &\left(\frac1{\mu(B(x,2\rho))}
          \int_{B(x,t_{j+1})}|\nabla u|^{\chi\gamma_j}\dmu\right)^{1/\chi}\\
 &\qquad\le C\gamma_j^2 4^j
          \frac1{\mu(B(x,2\rho))}\int_{B(x,t_j)}|\nabla u|^{\gamma_j}\dmu.
 \end{aligned}
\end{equation}
All integrals are normalized by the same volume $\mu(B(x,2\rho))$.
Taking the power $1/\gamma_j$ and using $\gamma_{j+1}=\chi\gamma_j$ gives
\[
 \begin{aligned}
 &\left(\frac1{\mu(B(x,2\rho))}
         \int_{B(x,t_{j+1})}|\nabla u|^{\gamma_{j+1}}\dmu\right)^{1/\gamma_{j+1}}\\
 &\qquad\le(C\gamma_j^2 4^j)^{1/\gamma_j}
      \left(\frac1{\mu(B(x,2\rho))}
         \int_{B(x,t_j)}|\nabla u|^{\gamma_j}\dmu\right)^{1/\gamma_j}.
 \end{aligned}
\]
Since
\[
 \sum_{j=0}^\infty\frac1{\gamma_j}
 =\frac\chi{(r+2q-m+2)(\chi-1)},\qquad
 \sum_{j=0}^\infty\frac j{\gamma_j}
 =\frac\chi{(r+2q-m+2)(\chi-1)^2},
\]
the corresponding infinite product converges and depends only on $n,m,q$.
Using the continuity of $|\nabla u|$ and $t_j\downarrow\rho$, we let $j\to\infty$ to obtain
\[
 \sup_{B(x,\rho)}|\nabla u|\le C
 \left(\frac1{\mu(B(x,2\rho))}
       \int_{B(x,3\rho/2)}|\nabla u|^{r+2q-m+2}\dmu\right)^{1/(r+2q-m+2)}.
\]
Applying~\eqref{eq:initial} gives
\begin{equation}\label{eq:small-ball}
 \sup_{B(x,\rho)}|\nabla u|
 \le C\left[\lambda^{-2}
       (K+\lambda^{-2/(q-m)}+\rho^{-2})\right]^{1/(2q-2m+2)}.
\end{equation}

Finally, by~\eqref{eq:scale},
\[
 \rho^{-2}=256\max\{R^{-2},K,\lambda^{-2/(q-m)}\}
 \le256(R^{-2}+K+\lambda^{-2/(q-m)}).
\]
Substituting this into~\eqref{eq:small-ball}, we find
\[
 |\nabla u(x)|\le C\lambda^{-1/(q-m+1)}
 \left[(K+R^{-2})^{1/(2q-2m+2)}+\lambda^{-1/[(q-m)(q-m+1)]}\right].
\]
Since
$\lambda^{-1/(q-m+1)}\lambda^{-1/[(q-m)(q-m+1)]}=\lambda^{-1/(q-m)}$
and $x\in B(o,R)$ was arbitrary, this proves~\eqref{eq:gradient-main}.
\end{proof}

\section{The range \texorpdfstring{$q>m$}{q > m}}\label{sec:global}

Let $q>m$ and let $u>0$ be a $C^1$ weak solution of~\eqref{eq:main}.

\begin{lemma}[Gradient control on high superlevel sets]\label{lem:strong-control}
There exists a constant $D>0$, depending only on $n,m,q$, such that
\begin{equation}\label{eq:strong-control}
 \lambda e^u|\nabla u|^{q-m}\le D\qquad\text{on }M.
\end{equation}
\end{lemma}

\begin{proof}
Taking $K=0$ in Proposition~\ref{thm:gradient} and letting $R\to\infty$ at an arbitrary center gives
\[
 \sup_M|\nabla u|\le C\lambda^{-1/(q-m)}.
\]
Fix an integer $k\ge0$.
All suprema over empty superlevel sets below are understood to be zero.
If $\sup_{\{u>k\}}|\nabla u|=0$, then
$\sup_{\{u>k+1\}}|\nabla u|=0$ as well.
Otherwise, we claim that $u(x)>k+1$ implies
\begin{equation}\label{eq:level-ball}
 B\left(x,\frac1{2\sup_{\{u>k\}}|\nabla u|}\right)
 \subset\{u>k\}.
\end{equation}
Suppose instead that a point $y$ in this ball satisfies $u(y)\le k$.
By the Hopf--Rinow theorem, there is a unit-speed minimizing geodesic
$\gamma:[0,d_g(x,y)]\to M$ joining $x$ to $y$.
Let $t_0$ be the first time at which $u(\gamma(t_0))=k$.
For $0\le t<t_0$, we have $u(\gamma(t))>k$, and hence
\[
 \begin{aligned}
 1<u(x)-k
 &=-\int_0^{t_0}\frac{d}{dt}u(\gamma(t))\,dt\\
 &\le\int_0^{t_0}|\nabla u(\gamma(t))|\,|\dot\gamma(t)|_g\,dt\\
 &\le d_g(x,y)\sup_{\{u>k\}}|\nabla u|<\frac12,
 \end{aligned}
\]
a contradiction. This proves~\eqref{eq:level-ball}.

On the ball in~\eqref{eq:level-ball}, the positive function $u-k$ satisfies
\[
 -\Delta_{m,g}(u-k)=\lambda e^k e^{u-k}|\nabla(u-k)|^q.
\]
Apply Proposition~\ref{thm:gradient} with radius
$(4\sup_{\{u>k\}}|\nabla u|)^{-1}$ and take the supremum over $u(x)>k+1$.
Absorbing the numerical factor into $C$, we obtain
\begin{equation}\label{eq:level-recursion}
 \begin{aligned}
 \sup_{\{u>k+1\}}|\nabla u|
 &\le C(\lambda e^k)^{-1/(q-m)}\\
 &\quad+C(\lambda e^k)^{-1/(q-m+1)}
       \left(\sup_{\{u>k\}}|\nabla u|\right)^{1/(q-m+1)}.
 \end{aligned}
\end{equation}
This also holds when the lower supremum is zero or the higher superlevel set is empty.
Multiplying by $(\lambda e^{k+1})^{1/(q-m)}$ and using
\[
 \frac1{q-m}-\frac1{q-m+1}
       -\frac1{(q-m)(q-m+1)}=0,
\]
we get
\begin{equation}\label{eq:normalized-recursion}
 \begin{aligned}
 &(\lambda e^{k+1})^{1/(q-m)}\sup_{\{u>k+1\}}|\nabla u|\\
 &\qquad\le C\left[1+
    \left((\lambda e^k)^{1/(q-m)}
          \sup_{\{u>k\}}|\nabla u|\right)^{1/(q-m+1)}\right].
 \end{aligned}
\end{equation}
Since $1/(q-m+1)<1$, a sufficiently large upper bound is preserved by this recurrence.
Together with the initial bound
$\lambda^{1/(q-m)}\sup_M|\nabla u|\le C$, induction yields
\[
 (\lambda e^k)^{1/(q-m)}\sup_{\{u>k\}}|\nabla u|
 \le C\qquad\text{for every integer }k\ge0.
\]

Given $x\in M$, choose $k=\lceil u(x)\rceil-1$, so that $k<u(x)\le k+1$.
Then
\[
 \begin{aligned}
 \lambda e^{u(x)}|\nabla u(x)|^{q-m}
 &\le e^{u(x)-k}
    \left((\lambda e^k)^{1/(q-m)}
          \sup_{\{u>k\}}|\nabla u|\right)^{q-m}\\
 &\le C.
 \end{aligned}
\]
This proves~\eqref{eq:strong-control}.
\end{proof}

\begin{lemma}\label{lem:boundedness}
Under the assumptions of this section, $u$ is bounded on $M$.
\end{lemma}

\begin{proof}
By~\eqref{eq:strong-control},
\[
 \bigl|\nabla e^{u/(q-m)}\bigr|
 =\frac1{q-m}e^{u/(q-m)}|\nabla u|
 \le\frac1{q-m}\left(\frac D\lambda\right)^{1/(q-m)}.
\]
Fix $o\in M$ and $x\ne o$.
By completeness and the Hopf--Rinow theorem, $o$ and $x$ can be joined by a unit-speed minimizing geodesic
$\gamma:[0,d_g(o,x)]\to M$.
Along this geodesic,
\[
 \begin{aligned}
 e^{u(x)/(q-m)}-e^{u(o)/(q-m)}
 &=\int_0^{d_g(o,x)}
   \left\langle\nabla e^{u/(q-m)}(\gamma(t)),\dot\gamma(t)\right\rangle_g\,dt\\
 &\le\frac1{q-m}\left(\frac D\lambda\right)^{1/(q-m)}d_g(o,x).
 \end{aligned}
\]
The same inequality is immediate for $x=o$.
Thus
\begin{equation}\label{eq:logarithmic-growth}
 e^{u(x)/(q-m)}\le e^{u(o)/(q-m)}
  +\frac1{q-m}\left(\frac D\lambda\right)^{1/(q-m)}d_g(o,x).
\end{equation}

Choose $\alpha>D/(m-1)$, depending only on $n,m,q$.
For $\varphi\in C_0^\infty(M)$, test the original equation with
$\alpha^{m-1}e^{\alpha(m-1)u}\varphi$.
Smooth approximation on a relatively compact neighborhood makes this test admissible in~\eqref{eq:weak}.
Using
\[
 |\nabla e^{\alpha u}|^{m-2}\nabla e^{\alpha u}
 =\alpha^{m-1}e^{\alpha(m-1)u}|\nabla u|^{m-2}\nabla u,
\]
we obtain
\begin{equation}\label{eq:composition}
 \Delta_{m,g}(e^{\alpha u})
 =\alpha^{m-1}e^{\alpha(m-1)u}|\nabla u|^m
 \left[(m-1)\alpha-\lambda e^u|\nabla u|^{q-m}\right]\ge0.
\end{equation}
This identity holds in the distributional sense, including at points where the gradient vanishes.
Thus $e^{\alpha u}$ is positive and $m$-subharmonic, whereas $u$ is positive and $m$-superharmonic.
Raising~\eqref{eq:logarithmic-growth} to the power $\alpha(q-m)$ gives
\begin{equation}\label{eq:polynomial-growth}
 \sup_{B(o,R)}e^{\alpha u}\le C_u(1+R)^{\alpha(q-m)},
\end{equation}
where $C_u$ may depend on $u(o)$ and $\lambda$, but is independent of $R$.

Applying~\eqref{eq:weak-harnack} to $u$ gives
\[
 \fint_{B(o,2R)}u^\beta\dmu
 \le C_H^\beta u(o)^\beta.
\]
Hence Chebyshev's inequality yields, for every $T>0$,
\begin{equation}\label{eq:level-measure}
 \frac{\mu(B(o,2R)\cap\{u>T\})}{\mu(B(o,2R))}
 \le C_H^\beta\left(\frac{u(o)}T\right)^\beta.
\end{equation}
Applying~\eqref{eq:subsolution-bound} to $e^{\alpha u}$ with $t=1$ and using~\eqref{eq:level-measure}, we obtain
\begin{equation}\label{eq:bounded-recursion}
 \begin{aligned}
 \sup_{B(o,R)}e^{\alpha u}
 &\le\frac{C_1}{\mu(B(o,2R))}
   \left[\int_{B(o,2R)\cap\{u\le T\}}e^{\alpha u}\dmu
   +\int_{B(o,2R)\cap\{u>T\}}e^{\alpha u}\dmu\right]\\
 &\le C_1e^{\alpha T}
   +C_1C_H^\beta\left(\frac{u(o)}T\right)^\beta
        \sup_{B(o,2R)}e^{\alpha u}.
 \end{aligned}
\end{equation}
Choose $T$ large enough, independently of $R$, that
\[
 C_1C_H^\beta\left(\frac{u(o)}T\right)^\beta
 \le2^{-\alpha(q-m)-1}.
\]
Iterating~\eqref{eq:bounded-recursion} $j$ times, we find
\[
 \sup_{B(o,R)}e^{\alpha u}
 \le C_1e^{\alpha T}\sum_{i=0}^{j-1}2^{-i[\alpha(q-m)+1]}
     +2^{-j[\alpha(q-m)+1]}\sup_{B(o,2^jR)}e^{\alpha u}.
\]
For fixed $R>0$,~\eqref{eq:polynomial-growth} and
$1+2^jR\le2^j(1+R)$ give
\[
 \begin{aligned}
 2^{-j[\alpha(q-m)+1]}\sup_{B(o,2^jR)}e^{\alpha u}
 &\le C_u\,2^{-j[\alpha(q-m)+1]}(1+2^jR)^{\alpha(q-m)}\\
 &\le C_u(1+R)^{\alpha(q-m)}2^{-j}\longrightarrow0.
 \end{aligned}
\]
The geometric sum is at most $2$.
Letting $j\to\infty$, we obtain
\[
 \sup_{B(o,R)}e^{\alpha u}\le2C_1e^{\alpha T}
 \qquad\text{for every }R>0.
\]
The bound is independent of $R$, so $u$ is bounded on $M$.
This bound may depend on $u(o)$ through $T$.
The gradient constants in Proposition~\ref{thm:gradient} and Lemma~\ref{lem:strong-control}
remain independent of the solution.
\end{proof}

\begin{proof}[Proof of Theorem~\ref{thm:liouville} for $q>m$]
By Lemma~\ref{lem:boundedness}, $u$ is bounded.
Choose $\alpha$ as in its proof.
The function $u-\inf_Mu$ is nonnegative and $m$-superharmonic, since
\[
 \Delta_{m,g}(u-\inf_Mu)=\Delta_{m,g}u\le0.
\]
By~\eqref{eq:composition},
\[
 \Delta_{m,g}\left(e^{\alpha(u-\inf_Mu)}-1\right)
 =e^{-\alpha(m-1)\inf_Mu}\Delta_{m,g}(e^{\alpha u})\ge0,
\]
so $e^{\alpha(u-\inf_Mu)}-1$ is nonnegative and $m$-subharmonic.
Since $u$ is bounded,
\begin{equation}\label{eq:comparison-h}
 \alpha(u-\inf_Mu)
 \le e^{\alpha(u-\inf_Mu)}-1
 \le\alpha e^{\alpha(\sup_Mu-\inf_Mu)}(u-\inf_Mu).
\end{equation}
Apply~\eqref{eq:subsolution-bound} to $e^{\alpha(u-\inf_Mu)}-1$ with $t=\beta$ and~\eqref{eq:weak-harnack} to $u-\inf_Mu$.
By~\eqref{eq:comparison-h},
\begin{equation}\label{eq:final-harnack}
 \begin{aligned}
 \sup_{B(o,R)}(u-\inf_Mu)
 &\le\alpha^{-1}\sup_{B(o,R)}\left(e^{\alpha(u-\inf_Mu)}-1\right)\\
 &\le C_\beta e^{\alpha(\sup_Mu-\inf_Mu)}
       \left(\fint_{B(o,2R)}(u-\inf_Mu)^\beta\dmu\right)^{1/\beta}\\
 &\le C_\beta C_H e^{\alpha(\sup_Mu-\inf_Mu)}
       \inf_{B(o,R)}(u-\inf_Mu).
 \end{aligned}
\end{equation}
The constant in~\eqref{eq:final-harnack} is independent of $R$.
For every $\varepsilon>0$, there is a point $x_\varepsilon\in M$ such that
$u(x_\varepsilon)<\inf_Mu+\varepsilon$.
Every sufficiently large ball $B(o,R)$ contains this point, so
\[
 \inf_{B(o,R)}(u-\inf_Mu)\downarrow0.
\]
Thus no attainment of the global infimum is needed.
Fixing an arbitrary point in~\eqref{eq:final-harnack} and letting $R\to\infty$ gives
$u\equiv\inf_Mu$. Hence $u$ is constant.
\end{proof}

\section{The range \texorpdfstring{$m-1<q<m$}{m-1 < q < m}}
\label{sec:remaining}

Let $m-1<q<m$ and let $u>0$ be a $C^1$ weak solution of~\eqref{eq:main}.
We use an auxiliary function and a global vanishing argument to prove constancy.

\subsection{A scalar cutoff function}

Fix
\begin{equation}\label{eq:new-parameters}
 \beta=\frac{(m-q)(n-1)}{4(n-1+|n-m|)}.
\end{equation}
Fix $\ell$, depending only on $n,m,q$, such that
\begin{equation}\label{eq:new-cap-parameters}
 0<\ell\le\min\left\{1,\frac{m-1}{4(m-q)(n-1)},\frac{q-m+1}{4(m-q)}\right\}.
\end{equation}
These parameters remain fixed throughout the section.

\begin{lemma}\label{lem:new-cap}
For $P>0$ and $0<\varepsilon\le1$, define
\begin{equation}\label{eq:new-cap-definition}
 \chi(P,\varepsilon)=(1+\varepsilon P^\ell)^{-2/\ell},
 \qquad
 \phi(P,\varepsilon)=
 \exp\left(\int_0^P
 \frac{(m-q)\chi(t,\varepsilon)}{1+(m-q)t\chi(t,\varepsilon)}\,dt\right).
\end{equation}
Then
\begin{align}
 &1\le\phi(P,\varepsilon)
 \le\min\{1+(m-q)P,C\varepsilon^{-2/\ell}\},
 \label{eq:new-cap-size}\\
 &\chi(P,\varepsilon)\le\varepsilon^{-2/\ell}P^{-2},
 \qquad |P\,\partial_{P}\chi|\le2(1-\chi),
 \label{eq:new-chi-est}
\end{align}
and $\phi(\infty,\varepsilon)<\infty$.
The parameter derivatives satisfy
\begin{equation}\label{eq:new-epsilon-derivatives}
 |(\varepsilon\partial_\varepsilon)\log\phi|+|(\varepsilon\partial_\varepsilon)^2\log\phi|
 +|P\partial_P(\varepsilon\partial_\varepsilon)\log\phi|
 \le C(1-\chi).
\end{equation}
Moreover, $P\partial_P\log\phi$, $(P\partial_P)^2\log\phi$, and
$(\varepsilon\partial_\varepsilon) P\partial_P\log\phi$ are uniformly bounded.
\end{lemma}

\begin{proof}
Since $0<\chi\le1$, we have
\[
 0\le
 \frac{(m-q)\chi(t,\varepsilon)}
 {1+(m-q)t\chi(t,\varepsilon)}
 \le\frac{m-q}{1+(m-q)t},
\]
which gives $1\le\phi\le1+(m-q)P$.
Moreover,
\[
 \chi(t,\varepsilon)\le\varepsilon^{-2/\ell}t^{-2}.
\]
Splitting the defining integral at $t=\varepsilon^{-2/\ell}$ gives
\[
 \begin{aligned}
 \log\phi(\infty,\varepsilon)
 &\le
 \log\bigl(1+(m-q)\varepsilon^{-2/\ell}\bigr)
 +(m-q)\varepsilon^{-2/\ell}
 \int_{\varepsilon^{-2/\ell}}^\infty t^{-2}\,dt\\
 &=\log\bigl(1+(m-q)\varepsilon^{-2/\ell}\bigr)+(m-q).
 \end{aligned}
\]
Thus $\phi(\infty,\varepsilon)\le C\varepsilon^{-2/\ell}<\infty$,
proving \eqref{eq:new-cap-size}.

Direct differentiation gives
\begin{equation}\label{eq:new-chi-derivatives}
 P\,\partial_{P}\chi(P,\varepsilon)=-\frac{2\varepsilon P^\ell}{1+\varepsilon P^\ell}\chi,
 \qquad
 \varepsilon\,\partial_{\varepsilon}\chi(P,\varepsilon)
 =-\frac{2\varepsilon P^\ell}{\ell(1+\varepsilon P^\ell)}\chi.
\end{equation}
Since $2/\ell\ge2$,
\[
 \frac{\varepsilon P^\ell}{1+\varepsilon P^\ell}
 \le1-(1+\varepsilon P^\ell)^{-2/\ell}
 =1-\chi.
\]
Combining this with $0<\chi\le1$ gives
$|P\partial_P\chi|\le2(1-\chi)$, proving \eqref{eq:new-chi-est}.

For the parameter derivatives, we first establish a uniform integral bound.
In expressions involving $t$, we write $\chi=\chi(t,\varepsilon)$.
By \eqref{eq:new-chi-derivatives},
\[
 \frac{d}{dt}\frac{1}{1+(m-q)t\chi}
 =
 -\frac{(m-q)\chi}{(1+(m-q)t\chi)^2}
 \frac{1-\varepsilon t^\ell}{1+\varepsilon t^\ell}.
\]
The last factor satisfies
\[
\begin{aligned}
\frac{1-\varepsilon t^\ell}{1+\varepsilon t^\ell}
&\ge\frac13,
&&\text{if }\varepsilon t^\ell\le\frac12,\\
\frac{1-\varepsilon t^\ell}{1+\varepsilon t^\ell}
&\le-\frac13,
&&\text{if }\varepsilon t^\ell\ge2.
\end{aligned}
\]
Since $0<(1+(m-q)t\chi)^{-1}\le1$, integration gives
\[
 \int_0^{(2\varepsilon)^{-1/\ell}}
 \frac{(m-q)\chi}{(1+(m-q)t\chi)^2}\,dt\le3,
 \qquad
 \int_{(2/\varepsilon)^{1/\ell}}^\infty
 \frac{(m-q)\chi}{(1+(m-q)t\chi)^2}\,dt\le3.
\]
On the intermediate interval,
\[
 \frac{(m-q)\chi}{(1+(m-q)t\chi)^2}\le\frac1{4t},
\]
and therefore
\[
 \int_{(2\varepsilon)^{-1/\ell}}^{(2/\varepsilon)^{1/\ell}}
 \frac{(m-q)\chi}{(1+(m-q)t\chi)^2}\,dt
 \le\frac{\log2}{2\ell}.
\]
Combining these estimates yields
\begin{equation}\label{eq:new-I-bound}
 \int_0^\infty
 \frac{(m-q)\chi(t,\varepsilon)}
 {(1+(m-q)t\chi(t,\varepsilon))^2}\,dt
 \le6+\frac{\log2}{2\ell}.
\end{equation}

Differentiating the defining integral gives
\begin{equation}\label{eq:new-epsilon-exact}
 \varepsilon\,\partial_\varepsilon\log\phi(P,\varepsilon)
 =-\frac2\ell\int_0^P
 \frac{(m-q)\chi}{(1+(m-q)t\chi)^2}
 \frac{\varepsilon t^\ell}{1+\varepsilon t^\ell}\,dt.
\end{equation}
A further differentiation, using \eqref{eq:new-chi-derivatives}, gives
\[
 \begin{aligned}
 &\left|
 \varepsilon\,\partial_\varepsilon
 \left(
 \frac{(m-q)\chi}{(1+(m-q)t\chi)^2}
 \frac{\varepsilon t^\ell}{1+\varepsilon t^\ell}
 \right)
 \right|
\le
 \left(1+\frac2\ell\right)
 \frac{(m-q)\chi}{(1+(m-q)t\chi)^2}
 \frac{\varepsilon t^\ell}{1+\varepsilon t^\ell}.
 \end{aligned}
\]
Since
\[
(\varepsilon\partial_\varepsilon)^2
=\varepsilon^2\partial_\varepsilon^2
+\varepsilon\partial_\varepsilon,
\]
this estimate and \eqref{eq:new-epsilon-exact} imply
\[
\begin{aligned}
&|\varepsilon\,\partial_\varepsilon\log\phi|
+|\varepsilon^{2}\,\partial_\varepsilon^2\log\phi|\\
&\qquad\le
C\int_0^P
\frac{(m-q)\chi(t,\varepsilon)}
{(1+(m-q)t\chi(t,\varepsilon))^2}
\frac{\varepsilon t^\ell}{1+\varepsilon t^\ell}\,dt\\
&\qquad\le
C\frac{\varepsilon P^\ell}{1+\varepsilon P^\ell}
\int_0^\infty
\frac{(m-q)\chi(t,\varepsilon)}
{(1+(m-q)t\chi(t,\varepsilon))^2}\,dt\\
&\qquad\le
C\left(6+\frac{\log2}{2\ell}\right)
\frac{\varepsilon P^\ell}{1+\varepsilon P^\ell}\\
&\qquad\le
C\left[1-(1+\varepsilon P^\ell)^{-2/\ell}\right]
=C(1-\chi(P,\varepsilon)).
\end{aligned}
\]
Differentiating \eqref{eq:new-epsilon-exact} with respect to its upper limit,
we also obtain
\[
 \begin{aligned}
 |P\partial_P(\varepsilon\partial_\varepsilon)\log\phi|
 &=
 \frac2\ell
 \frac{(m-q)P\chi}{(1+(m-q)P\chi)^2}
 \frac{\varepsilon P^\ell}{1+\varepsilon P^\ell}\\
 &\le C(1-\chi).
 \end{aligned}
\]
This proves \eqref{eq:new-epsilon-derivatives}.

Finally,
\begin{align}
 P\partial_P\log\phi
 &=\frac{(m-q)P\chi}{1+(m-q)P\chi},\label{Ppartialp}\\
 (P\partial_P)^2\log\phi
 &=\frac{(m-q)P\chi}{(1+(m-q)P\chi)^2}
   \frac{1-\varepsilon P^\ell}{1+\varepsilon P^\ell}.
\end{align}

Their absolute values are bounded by $1$ and $1/4$, respectively.
The remaining mixed derivative equals
$P\partial_P(\varepsilon\partial_\varepsilon)\log\phi$,
which has already been estimated.
\end{proof}

\subsection{Positivity of the scalar quadratic expression}

\begin{lemma}\label{lem:new-Q-positive}
Let $\beta$ and $\ell$ be given by
\eqref{eq:new-parameters}--\eqref{eq:new-cap-parameters}.
There exists a constant $c\ge1$, depending only on $n,m,q$,
such that the following conclusions hold. For $u\ge0$ and $P>0$, set
\[
 \varepsilon(u)=(u+c)^{-4}.
\]
In this lemma, $\chi$ and $\phi$ denote
$\chi(P,\varepsilon(u))$ and $\phi(P,\varepsilon(u))$, respectively.
We regard $u$ and $P$ as independent variables:
all $u$-derivatives are taken with $P$ fixed, including the dependence
of $\chi$ and $\phi$ on $\varepsilon(u)$. Define
\begin{equation}\label{eq:new-t-definition}
 t=t(u,P)=-2P\chi+
 \frac{2}{m-q}
 \left(\frac{\beta}{u+c}-\partial_u\log\phi\right)
 \bigl(1+(m-q)P\chi\bigr).
\end{equation}
Then
\begin{align}
 Q:={}&-\partial_u t
 -P\left(1-\frac{m-q}{2}t\right)\partial_P t
 -\frac{n-m}{2(n-1)}t^2
 \notag\\
 &+\left(m-q-1+\frac{2(m-1)}{n-1}\right)tP
 -2P+\frac{2(m-1)}{n-1}P^2
 \notag\\
 &\ge c_1\frac{(1+P)^2}{(u+c)^2},
 \label{eq:new-Q-positive}
\end{align}
where $c_1>0$ depends only on $n,m,q$.
Moreover, there exist $P_0\ge1$ and $c_2>0$, depending only on
$n,m,q$, such that
\begin{equation}\label{eq:new-Q-strong}
 Q\ge c_2\left\{
 \frac{1}{(u+c)^2}
 +\frac{P+P^2}{u+c}
 +P^2(1-\chi)
 \right\}
 \qquad(P\ge P_0).
\end{equation}
\end{lemma}

\begin{proof}
\proofstep{Step 1: Parameter derivative bounds.}

We derive the estimates below for a provisional value
\[
c\ge \max\{1,\beta\},
\]
and fix its final value after choosing $P_0$. Throughout the proof, $C>0$
denotes a constant independent of $u$, $P$, and $c$, whose value may
change from line to line.

Since
\[
\varepsilon(u)=(u+c)^{-4},
\qquad
\partial_u\varepsilon=-\frac{4\varepsilon}{u+c},
\]
the chain rule, with $\varepsilon=\varepsilon(u)$, gives
\begin{align}
\partial_u\log\phi
&=-\frac{4}{u+c}
  (\varepsilon\partial_\varepsilon)\log\phi,\label{eqlogphi}
\\
P\partial_P\partial_u\log\phi
&=-\frac{4}{u+c}
  P\partial_P\!\left[
  (\varepsilon\partial_\varepsilon)\log\phi
  \right],
\\
\partial_u^2\log\phi
&=\frac{4}{(u+c)^2}
  (\varepsilon\partial_\varepsilon)\log\phi
  +\frac{16}{(u+c)^2}
  (\varepsilon\partial_\varepsilon)^2\log\phi,
\\
\partial_u\chi
&=-\frac{4}{u+c}
  (\varepsilon\partial_\varepsilon)\chi.
\end{align}

Combining these identities with Lemma~\ref{lem:new-cap}, we obtain
\begin{equation}\label{eq:new-u-derivative-bounds}
\begin{aligned}
|\partial_u\log\phi|
+|P\partial_P\partial_u\log\phi|
+|\partial_u\chi|
&\le
\frac{C(1-\chi)}{u+c},
\\
|\partial_u^2\log\phi|
&\le
\frac{C(1-\chi)}{(u+c)^2}.
\end{aligned}
\end{equation}

\proofstep{Step 2: Expansion and error control.}

Set
\[
t_0=-2P\chi+
\frac{2\beta}{(m-q)(u+c)}\bigl(1+(m-q)P\chi\bigr).
\]
Then
\[
t-t_0=-\frac{2}{m-q}(\partial_u\log\phi)
\bigl(1+(m-q)P\chi\bigr).
\]
Using \eqref{eq:new-u-derivative-bounds},
$1+(m-q)P\chi\le C(1+P)$, and
$|\chi+P\partial_P\chi|\le3$, we obtain
\[
\begin{aligned}
&|t-t_0|+|P\partial_P(t-t_0)|
+(u+c)|\partial_u(t-t_0)|
\le\frac{C(1-\chi)(1+P)}{u+c},\\
&|t_0|+|P\partial_Pt_0|\le C(1+P).
\end{aligned}
\]
The identities
\[
\begin{aligned}
-\partial_ut_0
&=\frac{2\beta}{(m-q)(u+c)^2}
\bigl(1+(m-q)P\chi\bigr)
+2\left(1-\frac{\beta}{u+c}\right)P\partial_u\chi,\\
1-\frac{m-q}{2}t_0
&=\left(1-\frac{\beta}{u+c}\right)
\bigl(1+(m-q)P\chi\bigr),\\
\partial_Pt_0
&=-2\left(1-\frac{\beta}{u+c}\right)
(\chi+P\partial_P\chi)
\end{aligned}
\]
show how to separate the terms containing parameter derivatives.
Replacing $t$ by $t_0$ in $Q$ and omitting the term containing
$\partial_u\chi$ changes the expression by at most
\[
\frac{C(1-\chi)(1+P)^2}{u+c}.
\]
Expanding the remaining terms and collecting their explicit powers of
$(u+c)^{-1}$ gives
\begin{align}
Q\ge{}&
2P(\chi+P\partial_P\chi-1)
\notag\\
&+2P^2\left\{
(m-q)\chi P\partial_P\chi
+(1-\chi)\left[
\frac{m-1}{n-1}
+\left(
\frac{n-m}{n-1}-(m-q)
\right)\chi
\right]
\right\}
\notag\\
&+\frac{\beta P}{u+c}
\bigl(1+(m-q)P\chi\bigr)
\left[
-4(\chi+P\partial_P\chi)
+\frac{4(n-m)}{(m-q)(n-1)}\chi
\right.
\notag\\
&\hspace{40mm}\left.
+\frac{2}{m-q}
\left(
m-q-1+\frac{2(m-1)}{n-1}
\right)
\right]
\notag\\
&+\frac{\beta^2}{(u+c)^2}
\left[
2P(\chi+P\partial_P\chi)
\bigl(1+(m-q)P\chi\bigr)
+\frac{2\bigl(1+(m-q)P\chi\bigr)}
{(m-q)\beta}
\right.
\notag\\
&\hspace{40mm}\left.
-\frac{2(n-m)\bigl(1+(m-q)P\chi\bigr)^2}
{(m-q)^2(n-1)}
\right]
\notag\\
&-\frac{C(1-\chi)(1+P)^2}{u+c}.
\label{eq:new-Q-expanded}
\end{align}

\proofstep{Step 3: Positivity of the cutoff contribution.}

We first estimate the terms without an explicit factor
$(u+c)^{-1}$. Set
\[
a_0=
\min\left\{
\frac{m-1}{n-1},\,q-m+1
\right\}>0.
\]
By \eqref{eq:new-cap-parameters},
\[
(m-q)\ell\le\frac{a_0}{4}.
\]
From \eqref{eq:new-chi-derivatives} and
$\frac{x}{1+x}\le\log(1+x)$ for $x\ge0$, we obtain
\[
\begin{aligned}
|P\partial_P\chi|
&=
2\chi\frac{\varepsilon P^\ell}{1+\varepsilon P^\ell}
\\
&\le
2\chi\log(1+\varepsilon P^\ell)
=
\ell\chi(-\log\chi).
\end{aligned}
\]
Since
\[
\chi^2(-\log\chi)
\le\chi(1-\chi)
\le1-\chi,
\]
it follows that
\[
(m-q)\chi P\partial_P\chi
\ge
-(m-q)\ell(1-\chi).
\]
Moreover,
\[
\begin{aligned}
&\frac{m-1}{n-1}
+\left(
\frac{n-m}{n-1}-(m-q)
\right)\chi
\\
&\qquad=
(1-\chi)\frac{m-1}{n-1}
+\chi(q-m+1)
\ge a_0.
\end{aligned}
\]
Therefore,
\[
\begin{aligned}
&(m-q)\chi P\partial_P\chi
+(1-\chi)\left[
\frac{m-1}{n-1}
+\left(
\frac{n-m}{n-1}-(m-q)
\right)\chi
\right]
\\
&\qquad\ge
\bigl[a_0-(m-q)\ell\bigr](1-\chi)
\ge\frac{3a_0}{4}(1-\chi).
\end{aligned}
\]
Using $|P\partial_P\chi|\le2(1-\chi)$, we also have
\[
\begin{aligned}
2P(\chi+P\partial_P\chi-1)
&\ge
-2P\bigl[(1-\chi)+|P\partial_P\chi|\bigr]
\\
&\ge
-6P(1-\chi).
\end{aligned}
\]
Consequently,
\begin{equation}\label{eq:new-cutoff-lower-bound}
\begin{aligned}
&2P(\chi+P\partial_P\chi-1)
\\
&\quad
+2P^2\left\{
(m-q)\chi P\partial_P\chi
+(1-\chi)\left[
\frac{m-1}{n-1}
+\left(
\frac{n-m}{n-1}-(m-q)
\right)\chi
\right]
\right\}
\\
&\qquad\ge
(1-\chi)
\left(
\frac{3a_0}{2}P^2-6P
\right).
\end{aligned}
\end{equation}

\proofstep{Step 4: Choice of parameters and absorption.}

We now estimate the terms in \eqref{eq:new-Q-expanded} containing
$\beta/(u+c)$ or $\beta^2/(u+c)^2$. Since
\[
|\chi-1|+|P\partial_P\chi|
\le3(1-\chi)
\]
and
\[
\left|
\bigl(1+(m-q)P\chi\bigr)
-\bigl(1+(m-q)P\bigr)
\right|
=(m-q)P(1-\chi),
\]
replacing $\chi$ by $1$ and $P\partial_P\chi$ by $0$ in these
terms produces an error bounded in absolute value by
\[
\frac{C(1-\chi)(1+P)^2}{u+c}.
\]
Using
\[
-4+\frac{4(n-m)}{(m-q)(n-1)}
+\frac{2}{m-q}
\left(
m-q-1+\frac{2(m-1)}{n-1}
\right)
=
\frac{2(q-m+1)}{m-q},
\]
and collecting the constant, linear, and quadratic terms in $P$, we
write the terms after this replacement as
\[
\begin{aligned}
&\frac{2\beta^2}{(m-q)^2(u+c)^2}
\left(
\frac{m-q}{\beta}
-\frac{n-m}{n-1}
\right)
\\
&\quad
+\left[
\frac{2\beta(q-m+1)}{(m-q)(u+c)}
+\frac{\beta^2}{(u+c)^2}
\left(
2+\frac{2}{\beta}
-\frac{4(n-m)}{(m-q)(n-1)}
\right)
\right]P
\\
&\quad
+\left[
\frac{2\beta(q-m+1)}{u+c}
-\frac{2\beta^2}{(u+c)^2}
\left(
\frac{n-m}{n-1}-(m-q)
\right)
\right]P^2.
\end{aligned}
\]

By the definition of $\beta$,
\[
\frac{m-q}{\beta}
=
4\left(
1+\frac{|n-m|}{n-1}
\right)
>
\frac{n-m}{n-1},
\]
so the constant term is positive. Since $q-m+1>0$, the terms of
order $(u+c)^{-1}$ in the coefficients of $P$ and $P^2$ are also
positive. Fix $b_0>0$, depending only on $n,m,q$, such that
\[
b_0\le\min\left\{
\frac{\beta^2}{(m-q)^2}
\left(\frac{m-q}{\beta}-\frac{n-m}{n-1}\right),\,
\frac{\beta(q-m+1)}{m-q},\,
\beta(q-m+1)
\right\}.
\]
For every sufficiently large $c$, the terms of order $(u+c)^{-2}$
in the coefficients of $P$ and $P^2$ can be absorbed into half of
their positive terms of order $(u+c)^{-1}$. The preceding expression
is therefore bounded below by
\[
b_0\left[
\frac{1}{(u+c)^2}
+\frac{P+P^2}{u+c}
\right].
\]
Combining this estimate with
\eqref{eq:new-cutoff-lower-bound} and
\eqref{eq:new-Q-expanded}, and absorbing the replacement error, gives
\begin{equation}\label{eq:new-Q-combined-bound}
\begin{aligned}
Q\ge{}&
b_0\left[
\frac{1}{(u+c)^2}
+\frac{P+P^2}{u+c}
\right]
\\
&+(1-\chi)
\left(
\frac{3a_0}{2}P^2-6P
\right)
-\frac{C(1-\chi)(1+P)^2}{u+c}.
\end{aligned}
\end{equation}

With $\beta$, $\ell$, $a_0$, and $b_0$ fixed, choose
$P_0\ge\max\{1,8/a_0\}$, independently of $c$, so that
\[
6P\le\frac{3a_0}{4}P^2
\qquad(P\ge P_0).
\]
Since $(1+P)^2\le4P^2$ for $P\ge1$, imposing the additional
lower bound on $c$ given by
\[
\frac{4C}{c}\le\frac{3a_0}{8}
\]
gives
\[
\begin{aligned}
Q\ge{}&
b_0\left[
\frac{1}{(u+c)^2}
+\frac{P+P^2}{u+c}
\right]
\\
&+\frac{3a_0}{8}P^2(1-\chi)
\qquad(P\ge P_0).
\end{aligned}
\]
Thus \eqref{eq:new-Q-strong} holds with
\[
c_2=
\min\left\{
b_0,\frac{3a_0}{8}
\right\}>0.
\]

It remains to prove \eqref{eq:new-Q-positive} for bounded $P$.
If $0<P\le P_0$, then
\[
\begin{aligned}
1-\chi
&=
\frac{2}{\ell}
\int_0^{\varepsilon P^\ell}
(1+s)^{-2/\ell-1}\,ds
\\
&\le
\frac{2}{\ell}\varepsilon P^\ell
\le
\frac{C}{(u+c)^4}.
\end{aligned}
\]
Hence the terms in \eqref{eq:new-Q-combined-bound} containing
$1-\chi$ are bounded below by $-C(u+c)^{-4}$. Dropping the
nonnegative term $b_0(P+P^2)/(u+c)$, we obtain
\[
Q\ge
\frac{b_0}{(u+c)^2}
-\frac{C}{(u+c)^4}.
\]
Imposing also $C/c^2\le b_0/2$, with this $C$ now depending on the
fixed $P_0$ but still independent of $c$, gives
\[
Q\ge
\frac{b_0}{2(u+c)^2}
\ge
\frac{b_0}{2(1+P_0)^2}
\frac{(1+P)^2}{(u+c)^2}
\qquad(0<P\le P_0).
\]

Finally, for $P\ge P_0$, \eqref{eq:new-Q-strong} and $u+c\ge1$
give
\[
\begin{aligned}
Q
&\ge
c_2\left[
\frac{1}{(u+c)^2}
+\frac{P+P^2}{u+c}
\right]
\\
&\ge
c_2\frac{1+P+P^2}{(u+c)^2}
\\
&\ge
\frac{c_2}{2}
\frac{(1+P)^2}{(u+c)^2}.
\end{aligned}
\]
Combining the two ranges proves \eqref{eq:new-Q-positive}, with
\[
c_1=
\min\left\{
\frac{b_0}{2(1+P_0)^2},
\frac{c_2}{2}
\right\}>0.
\]
Finally, fix $c$ large enough to satisfy all the preceding lower bounds.
The parameters $\beta$, $\ell$, $a_0$, and $b_0$ were chosen first,
then $P_0$, and only then $c$. Each choice depends only on $n,m,q$,
so the choices are independent of $u$ and $P$.
\end{proof}

\subsection{From the Bochner calculation to an unweighted divergence inequality}

Fix the parameter $c$ chosen in Lemma~\ref{lem:new-Q-positive}, and write $s=|\nabla u|$. On the
noncritical set $\{s>0\}$, define
\begin{equation}
\label{eq:new-auxiliary} P=\frac{\lambda e^u}{(m-1)s^{m-q}},\qquad \mathcal F=(u+c)^{-\beta}s^{m-q}\phi\bigl(P,(u+c)^{-4}\bigr).
\end{equation}
Set $\mathcal F=0$ on $\{s=0\}$. On each compact set, \eqref{eq:new-cap-size} bounds
$\phi(P,(u+c)^{-4})$ uniformly in $P$. Since $m-q>0$, this extension is continuous at critical
points. Lemma~\ref{lem:regularity} justifies the pointwise calculations on $\{s>0\}$.

\begin{lemma}\label{lem:new-divergence}

Let
\[
\A=I+(m-2)\frac{\nabla u}{s}\otimes\frac{\nabla u}{s}.
\]
There exist constants $c_0,C_0,J>0$, depending only on $n,m,q,\lambda$, such that, on the
noncritical set,
\begin{equation}
\label{eq:new-autonomous-divergence} \operatorname{div}(\A\nabla \mathcal F) \ge c_0\mathcal F^{2/(m-q)} -C_0(1+\log_+\mathcal F)^J\frac{|\nabla \mathcal F|^2}{\mathcal F}.
\end{equation}
Here $\log_+t=\max\{\log t,0\}$. The ellipticity constants of $\A$ are $\min\{1,m-1\}$ and
$\max\{1,m-1\}$, independently of the solution.

\end{lemma}

\begin{proof}

Throughout this proof, $c$ remains fixed. The symbols $c_*$ and $C$ denote positive constants
depending only on $n,m,q,\lambda$, whose values may change from line to line.

Since
\[
\A\frac{\nabla u}{s}=(m-1)\frac{\nabla u}{s},
\qquad \A\xi=\xi\quad\text{for }\xi\perp\nabla u,
\]
the eigenvalues of $\A$ are $m-1$ in the direction of $\nabla u$ and $1$
in the orthogonal directions.
Hence
\[
\min\{1,m-1\}I \le \A\le \max\{1,m-1\}I.
\]

\proofstep{Step 1: Derivatives of the auxiliary function.}

Set
\[
\rho=\log(s^2), \qquad \varepsilon=(u+c)^{-4}.
\]
Then
\[
s^{m-q}=e^{\frac{m-q}{2}\rho},\quad P =\frac{\lambda e^u}{(m-1)s^{m-q}} =\frac{\lambda}{m-1}
e^{\,u-\frac{m-q}{2}\rho}.
\]
Hence
\[
\partial_uP=P, \qquad \partial_\rho P=-\frac{m-q}{2}P, \qquad
\partial_u\varepsilon=-\frac{4}{u+c}\varepsilon.
\]
By the definition of $\mathcal F$, we have
\[
\log\mathcal F = -\beta\log(u+c) +\frac{m-q}{2}\rho +\log\phi(P,\varepsilon).
\]
Differentiating with respect to $\rho$ gives
\[
\begin{aligned}
\partial_\rho\log\mathcal F &= \frac{m-q}{2} -\frac{m-q}{2}P\partial_P\log\phi.
\end{aligned}
\]
By \eqref{Ppartialp} in Lemma~\ref{lem:new-cap},
\begin{equation}
\label{eq:new-kappa} \kappa := \frac{\mathcal F_\rho}{\mathcal F} = \partial_\rho\log\mathcal F
= \frac{m-q} {2\bigl(1+(m-q)P\chi\bigr)}.
\end{equation}
Moreover,
\[
P\chi = P(1+\varepsilon P^\ell)^{-2/\ell} \le 4^{-1/\ell}\varepsilon^{-1/\ell},
\]
so, since $0<\varepsilon\le1$,
\begin{equation}
\label{eqkappa} c_*\varepsilon^{1/\ell} \le\kappa\le\frac{m-q}{2}.
\end{equation}

Next,
\[
\begin{aligned}
\partial_u\log\mathcal F ={}& -\frac{\beta}{u+c} +P\partial_P\log\phi -\frac{4}{u+c}
\varepsilon\,\partial_\varepsilon\log\phi.
\end{aligned}
\]
Therefore,
\[
\begin{aligned}
-\frac{\mathcal F_u}{\mathcal F_\rho} &= -\frac{\partial_u\log\mathcal F}{\kappa}\\
&= -2P\chi +\frac{2}{m-q} \left[ \frac{\beta}{u+c} +\frac{4}{u+c}
\varepsilon\,\partial_\varepsilon\log\phi \right] \bigl(1+(m-q)P\chi\bigr).
\end{aligned}
\]
By \eqref{eq:new-t-definition} and \eqref{eqlogphi} in Lemma~\ref{lem:new-Q-positive},
\begin{equation}
\label{eqt} -\frac{\mathcal F_u}{\mathcal F_\rho}=t.
\end{equation}
In Lemma~\ref{lem:new-Q-positive}, $u$ and $P$ are independent variables. After composition with
$P=P(u,\rho)$, the chain rule gives
\[
\begin{aligned}
\frac{\partial}{\partial u}t(u,P(u,\rho)) &=\partial_ut+P\partial_Pt,\\
\frac{\partial}{\partial\rho}t(u,P(u,\rho)) &=-\frac{m-q}{2}P\partial_Pt.
\end{aligned}
\]
Consequently,
\begin{equation}
\label{eq:new-t-chain}
\begin{aligned}
&\frac{\partial}{\partial u}t(u,P(u,\rho)) +t\frac{\partial}{\partial\rho}t(u,P(u,\rho))\\
&\qquad=\partial_ut +P\left(1-\frac{m-q}{2}t\right)\partial_Pt.
\end{aligned}
\end{equation}
This is precisely the derivative combination appearing in the definition of $Q$.

To estimate the second derivatives of $\log\mathcal F$, observe that
$\partial_\rho=-\frac{m-q}{2}P\partial_P$ on functions of $(P,\varepsilon)$, with $u$ fixed. In
particular, $(P\partial_P)^2=P^2\partial_P^2+P\partial_P$. Differentiation gives
\[
\partial_{\rho\rho}\log\mathcal F = \frac{(m-q)^2}{4} (P\partial_P)^2\log\phi,
\]
and
\[
\partial_{u\rho}\log\mathcal F = -\frac{m-q}{2} \left[ (P\partial_P)^2\log\phi -\frac{4}{u+c}
P\partial_P \bigl((\varepsilon\partial_\varepsilon)\log\phi\bigr) \right].
\]
Similarly,
\[
\begin{aligned}
\partial_{uu}\log\mathcal F ={}& \frac{\beta}{(u+c)^2} +(P\partial_P)^2\log\phi -\frac{8}{u+c}
P\partial_P \bigl((\varepsilon\partial_\varepsilon)\log\phi\bigr)\\
&+\frac{4}{(u+c)^2} (\varepsilon\partial_\varepsilon)\log\phi +\frac{16}{(u+c)^2}
(\varepsilon\partial_\varepsilon)^2\log\phi.
\end{aligned}
\]
Lemma~\ref{lem:new-cap} gives
\[
|P\partial_P\log\phi| +\left|(P\partial_P)^2\log\phi\right|\le C
\]
and
\[
\begin{aligned}
& \left| (\varepsilon\partial_\varepsilon)\log\phi \right| + \left|
(\varepsilon\partial_\varepsilon)^2\log\phi \right| + \left| P\partial_P
\bigl((\varepsilon\partial_\varepsilon)\log\phi\bigr) \right| \\
&\qquad\le C(1-\chi)\le C.
\end{aligned}
\]
Together with the first- and second-derivative identities above and $u+c\ge1$, these estimates
yield
\begin{equation}
\label{eq:new-logF-derivative-bounds}
\begin{aligned}
&|\partial_u\log\mathcal F|+|\partial_\rho\log\mathcal F|
+|\partial_{uu}\log\mathcal F|\\
&\qquad+|\partial_{u\rho}\log\mathcal F|
+|\partial_{\rho\rho}\log\mathcal F|\le C.
\end{aligned}
\end{equation}
Finally, we have the identities
\[
\begin{aligned}
t&=-\frac{\partial_u\log\mathcal F}{\kappa},\\
\frac{\mathcal F_{\rho\rho}}{\mathcal F_\rho} &=\kappa+\frac{\partial_{\rho\rho}\log\mathcal F}{\kappa},\\
\frac{\mathcal F_{u\rho}}{\mathcal F_\rho} &=\partial_u\log\mathcal F
+\frac{\partial_{u\rho}\log\mathcal F}{\kappa}.
\end{aligned}
\]
Using \eqref{eqkappa} and \eqref{eq:new-logF-derivative-bounds}, we therefore obtain
\begin{equation}
\label{eq:new-ratio-bounds} |t| +\left|\frac{\mathcal F_{\rho\rho}}{\mathcal F_\rho}-1\right|
+\left|\frac{\mathcal F_{u\rho}}{\mathcal F_\rho}\right| \le C\varepsilon^{-1/\ell}.
\end{equation}
In particular, $\mathcal F_\rho=\kappa\mathcal F>0$ on the noncritical set.

\proofstep{Step 2: Hessian identities in an adapted frame.}

We rewrite the Hessian of $u$ in terms of $P$, $t$, and $\nabla\mathcal F$, using the identities
from Lemma~\ref{lem:hessian}. Set $E=\nabla\mathcal F/\mathcal F_\rho$. The chain rule
$\nabla\mathcal F=\mathcal F_u\nabla u+\mathcal F_\rho\nabla\rho$ and~\eqref{eqt} give
\begin{equation}
\label{eq:new-rho-gradient} \nabla\rho=t\nabla u+E.
\end{equation}
At a noncritical point, use the adapted frame and notation $u_{ij}$ of Lemma~\ref{lem:hessian},
with $e_1=\nabla u/s$. By~\eqref{eq:trace} and the definition of $P$,
\begin{equation}
\label{eq:new-trace} (m-1)u_{11}+\sum_{i=2}^n u_{ii}=-(m-1)Ps^2.
\end{equation}
Since $\rho=\log(s^2)$, identity~\eqref{eq:gradient-frame} gives
\begin{equation}
\label{rho} \nabla\rho =\frac{2u_{11}}s e_1+\sum_{j=2}^n\frac{2u_{1j}}s e_j.
\end{equation}
Comparing this with~\eqref{eq:new-rho-gradient}, namely
\[
\nabla\rho=(ts+E_1)e_1+\sum_{j=2}^n E_j e_j,
\]
we obtain
\begin{equation}
\label{eq:new-hessian-gradient} u_{11}=\frac12(ts^2+sE_1), \qquad u_{1j}=\frac
s2E_j\quad(j\ge2).
\end{equation}
Finally, the transverse estimate~\eqref{eq:transverse-hessian}, with $\lambda
e^us^{q-m+2}=(m-1)Ps^2$, yields
\begin{equation}
\label{eq:new-transverse-hessian} \sum_{i,j=2}^n u_{ij}^2 \ge\frac{(m-1)^2}{n-1}(Ps^2+u_{11})^2.
\end{equation}

The transverse Hessian estimate retains the cross term in
\[
(Ps^2+u_{11})^2 = P^2s^4+2Ps^2u_{11}+u_{11}^2.
\]
After substitution of \eqref{eq:new-hessian-gradient}, this term contributes to the $tP$
coefficient in $Q$.

\proofstep{Step 3: The Bochner calculation and the appearance of $Q$.}

We insert these Hessian identities into the Bochner formula to identify the terms independent of
$E$. The Bochner formula gives
\begin{equation}
\label{eqbochner} \frac12\Delta(s^2) = |\nabla^2u|^2 +\langle\nabla u,\nabla\Delta u\rangle
+\operatorname{Ric}(\nabla u,\nabla u).
\end{equation}
Since $e_1=\nabla u/s$,
\begin{equation}
\label{eqAnabla} \A:\nabla^2(s^2) = \Delta(s^2) +(m-2)\nabla^2(s^2)(e_1,e_1).
\end{equation}
At the chosen point,
\[
\frac12\nabla^2(s^2)(e_1,e_1) = \sum_{j=1}^n u_{1j}^2 +s\,\nabla^3u(e_1,e_1,e_1).
\]
We differentiate the normalized equation
\[
\Delta u+(m-2)u_{11} =-(m-1)Ps^2
\]
in the direction $\nabla u=se_1$. Since
\[
\nabla_{\nabla u/s}(\nabla u/s) = \frac1s\sum_{j=2}^n u_{1j}e_j,
\]
we have
\[
e_1(u_{11}) = \nabla^3u(e_1,e_1,e_1) +\frac2s\sum_{j=2}^n u_{1j}^2.
\]
Consequently,
\begin{align}
\label{eqexpand} & \langle\nabla u,\nabla\Delta u\rangle +(m-2)s\nabla^3u(e_1,e_1,e_1)\notag\\
&\qquad = -(m-1)\langle\nabla u,\nabla(Ps^2)\rangle -2(m-2)\sum_{j=2}^n u_{1j}^2.
\end{align}
Differentiating $P=\frac{\lambda}{m-1}e^{\,u-\frac{m-q}{2}\rho}$ gives
\[
\nabla P = P\left( \nabla u -\frac{m-q}{2}\nabla\rho \right) = P\left( \nabla u -\frac{m-q}{2}
\frac{\nabla(s^2)}{s^2} \right).
\]
Therefore,
\begin{align}
\langle\nabla u,\nabla(Ps^2)\rangle &= s^2\langle\nabla u,\nabla P\rangle +P\langle\nabla u,\nabla(s^2)\rangle\notag\\
&= Ps^4 +\frac{q-m+2}{2}P \langle\nabla u,\nabla(s^2)\rangle.\label{eqPs}
\end{align}
Combining \eqref{eqbochner}--\eqref{eqPs} and expanding $|\nabla^2u|^2$, we obtain
\begin{align}
\A:\nabla^2(s^2) ={}& 2(m-1)u_{11}^2 +2(4-m)\sum_{j=2}^n u_{1j}^2 +2\sum_{i,j=2}^n u_{ij}^2
\notag\\
&+ 2\operatorname{Ric}(\nabla u,\nabla u) -2(m-1)Ps^4 \notag\\
&- (q-m+2)(m-1)P \langle\nabla u,\nabla(s^2)\rangle. \label{eq:new-bochner-normalized}
\end{align}

We now apply the chain rule to $\mathcal F=\mathcal F(u,\rho)$. Writing $\A(X,Y)=\langle\A
X,Y\rangle$, we have
\[
\begin{aligned}
\A:\nabla^2\mathcal F ={}& \mathcal F_u\,\A:\nabla^2u +\mathcal F_\rho\,\A:\nabla^2\rho\\
&+ \mathcal F_{uu}\A(\nabla u,\nabla u) +2\mathcal F_{u\rho}\A(\nabla u,\nabla\rho)\\
&+ \mathcal F_{\rho\rho}\A(\nabla\rho,\nabla\rho).
\end{aligned}
\]
Since $\nabla(s^2)=s^2\nabla\rho$, we have
\[
\begin{aligned}
\A:\nabla^2\rho &=\frac{\A:\nabla^2(s^2)}{s^2} -\A(\nabla\rho,\nabla\rho),\\
\A:\nabla^2u &=-(m-1)Ps^2, \qquad \A(\nabla u,\nabla u)=(m-1)s^2.
\end{aligned}
\]
Moreover, differentiating \eqref{eqt} and using \eqref{eq:new-t-chain} gives
\[
\frac{\mathcal F_{uu}+2t\mathcal F_{u\rho} +t^2\mathcal F_{\rho\rho}}{\mathcal F_\rho}
=-\partial_ut -P\left(1-\frac{m-q}{2}t\right)\partial_Pt.
\]
Substituting these identities together with \eqref{eq:new-hessian-gradient} and
\eqref{eq:new-bochner-normalized} into the chain rule, we collect the $E$-terms and use the
definition of $Q$, retaining the transverse Hessian remainder and the Ricci term. This gives the
following identity.

\begin{equation}
\label{eq:new-offcritical-nondiv}
\begin{aligned}
\frac{\A:\nabla^2\mathcal F}{\mathcal F_\rho} ={}&(m-1)s^2Q +2(m-1)\frac{\mathcal F_{u\rho}}{\mathcal F_\rho}sE_1\\
&+(m-1)\left[ \left(2\frac{\mathcal F_{\rho\rho}}{\mathcal F_\rho} -\frac{n-m}{n-1}\right)t
+\left(\frac{2(m-1)}{n-1}-q+m-2\right)P \right]sE_1\\
&+(m-1)\left( \frac{\mathcal F_{\rho\rho}}{\mathcal F_\rho} -\frac{n-m}{2(n-1)}\right)E_1^2
+\left(\frac{\mathcal F_{\rho\rho}}{\mathcal F_\rho} -\frac{m-2}{2}\right)\sum_{j=2}^nE_j^2\\
&+\frac{2}{s^2}\left[ \sum_{i,j=2}^n u_{ij}^2 -\frac{(m-1)^2}{n-1}(Ps^2+u_{11})^2
+\operatorname{Ric}(\nabla u,\nabla u) \right].
\end{aligned}
\end{equation}
The last bracket is nonnegative by \eqref{eq:new-transverse-hessian} and
$\operatorname{Ric}\ge0$; we discard it in the lower bound below.

\proofstep{Step 4: Passage to divergence form and absorption of the error terms.}

To pass from \eqref{eq:new-offcritical-nondiv} to divergence form, we must include the
derivatives of $\A$. Differentiating its definition gives
\[
\operatorname{div}\A
=(m-2)\left[
\operatorname{div}\left(\frac{\nabla u}{s}\right)\frac{\nabla u}{s}
+\nabla_{\nabla u/s}\left(\frac{\nabla u}{s}\right)
\right].
\]
The identity $\operatorname{div}(\nabla u/s)=s^{-1}\sum_{i=2}^n u_{ii}$, together with
\eqref{eq:new-trace} and \eqref{eq:new-hessian-gradient}, gives
\[
\operatorname{div}(\nabla u/s) = -(m-1)Ps -\frac{m-1}{2}(ts+E_1).
\]
Furthermore,
\[
\nabla_{\nabla u/s}(\nabla u/s) = \frac1s\sum_{j=2}^n u_{1j}e_j = \frac12\sum_{j=2}^nE_je_j.
\]
Therefore
\begin{equation}
\label{eq:new-div-A} \operatorname{div}\A = (m-2) \left\{ -(m-1)Ps\,e_1
-\frac{m-1}{2}(ts+E_1)e_1 +\frac12\sum_{j=2}^nE_je_j \right\}.
\end{equation}
Since $ \operatorname{div}(\A\nabla\mathcal F) = \A:\nabla^2\mathcal F +
\langle\operatorname{div}\A,\nabla\mathcal F\rangle $ and $\nabla\mathcal F=\mathcal F_\rho E$,
combining \eqref{eq:new-offcritical-nondiv} and \eqref{eq:new-div-A} gives
\begin{align}
\frac{\operatorname{div}(\A\nabla\mathcal F)} {\mathcal F_\rho} \ge{}& (m-1)s^2Q \notag\\
&+ (m-1) \Bigg\{ \left[ \frac{4-m}{2} +\frac{m-1}{n-1} +2\left( \frac{\mathcal F_{\rho\rho}}{\mathcal F_\rho}-1 \right) \right]t \notag\\
&\qquad\qquad + \left( \frac{2(m-1)}{n-1}-q \right)P + 2\frac{\mathcal F_{u\rho}}{\mathcal F_\rho} \Bigg\} sE_1 \notag\\
&+ (m-1) \left[ \frac{3-m}{2} +\frac{m-1}{2(n-1)} +\frac{\mathcal F_{\rho\rho}}{\mathcal F_\rho}-1 \right]E_1^2 \notag\\
&+ \frac{\mathcal F_{\rho\rho}}{\mathcal F_\rho} \sum_{j=2}^nE_j^2.
\label{eq:new-offcritical-div}
\end{align}

By \eqref{eq:new-ratio-bounds}, we have
\[
\begin{aligned}
&\Bigg|(m-1)\Bigg\{ \left[ \frac{4-m}{2}+\frac{m-1}{n-1} +2\left(\frac{\mathcal F_{\rho\rho}}{\mathcal F_\rho}-1\right) \right]t\\
&\qquad +\left(\frac{2(m-1)}{n-1}-q\right)P +2\frac{\mathcal F_{u\rho}}{\mathcal F_\rho}
\Bigg\}sE_1\Bigg| \le C\varepsilon^{-2/\ell}(1+P)s|E|,
\end{aligned}
\]
and
\[
\begin{aligned}
&\Bigg| (m-1)\left[ \frac{3-m}{2}+\frac{m-1}{2(n-1)} +\frac{\mathcal F_{\rho\rho}}{\mathcal F_\rho}-1 \right]E_1^2 +\frac{\mathcal F_{\rho\rho}}{\mathcal F_\rho} \sum_{j=2}^nE_j^2
\Bigg|\le C\varepsilon^{-1/\ell}|E|^2.
\end{aligned}
\]
Fix $\delta=(m-1)/4$. Lemma~\ref{lem:new-Q-positive} and Young's inequality give
\[
\begin{aligned}
C\varepsilon^{-2/\ell}(1+P)s|E| &\le \delta s^2Q + C\varepsilon^{-4/\ell}
\frac{(1+P)^2}{Q}|E|^2\\
&\le \delta s^2Q + C\varepsilon^{-4/\ell} (u+c)^2|E|^2.
\end{aligned}
\]
This choice retains a fixed positive portion of the $s^2Q$ term. Since $E=\nabla\mathcal F/(\kappa\mathcal F)$, multiplication by $\mathcal F_\rho=\kappa\mathcal F$ gives
\[
\begin{aligned}
\operatorname{div}(\A\nabla\mathcal F) \ge{}& c_*\kappa\mathcal F s^2Q -
C\varepsilon^{-4/\ell}(u+c)^2 \frac{|\nabla\mathcal F|^2}{\kappa\mathcal F}\\
&- C\varepsilon^{-1/\ell} \frac{|\nabla\mathcal F|^2}{\kappa\mathcal F}.
\end{aligned}
\]
Using $\kappa^{-1}\le C\varepsilon^{-1/\ell}$ and $u+c\ge1$, we bound the sum of the two
negative terms in absolute value by
\[
C\varepsilon^{-5/\ell}(u+c)^2 \frac{|\nabla\mathcal F|^2}{\mathcal F}.
\]
Since $\varepsilon=(u+c)^{-4}$, we obtain
\begin{equation}
\label{eq:new-coarse-div} \operatorname{div}(\A\nabla\mathcal F) \ge c_*\kappa\mathcal F s^2Q -
C(u+c)^{20/\ell+2} \frac{|\nabla\mathcal F|^2}{\mathcal F}.
\end{equation}

\proofstep{Step 5: Coercivity of the positive term.}

The definition of $\mathcal F$ and the identity $s^{m-q}=\lambda e^u/((m-1)P)$ give
\begin{equation}
\label{eq:new-coercion-exact} \frac{\kappa s^2Q} {\mathcal F^{(q-m+2)/(m-q)}} =
\frac{\lambda}{m-1} e^u (u+c)^{\beta(q-m+2)/(m-q)} \frac{\kappa Q} {P\phi^{(q-m+2)/(m-q)}}.
\end{equation}
By \eqref{eqkappa} and Lemmas~\ref{lem:new-cap} and~\ref{lem:new-Q-positive},
\[
\begin{aligned}
\frac{\kappa Q} {P\phi^{(q-m+2)/(m-q)}} &\ge c_*\, \frac{(1+P)^2}{P(u+c)^2} \varepsilon^{
\frac1\ell + \frac{2(q-m+2)}{(m-q)\ell} }\\
&= c_*\, \frac{(1+P)^2}{P(u+c)^2} \varepsilon^{ \frac{q-m+4}{(m-q)\ell} }.
\end{aligned}
\]
Since $(1+P)^2/P\ge4$, \eqref{eq:new-coercion-exact} yields
\[
\begin{aligned}
\frac{\kappa s^2Q} {\mathcal F^{(q-m+2)/(m-q)}} &\ge c_* e^u (u+c)^{ \beta(q-m+2)/(m-q)-2 }
\varepsilon^{ \frac{q-m+4}{(m-q)\ell} }\\
&\ge c_* e^u (u+c)^{ -2-\frac{4(q-m+4)}{(m-q)\ell} }.
\end{aligned}
\]
In the last inequality we used $u+c\ge1$ and $\varepsilon=(u+c)^{-4}$. The function
\[
e^u (u+c)^{ -2-\frac{4(q-m+4)}{(m-q)\ell} }
\]
has a positive infimum for $u\ge0$: it is positive and continuous, and the exponential factor
dominates the fixed negative power as $u\to\infty$. Hence
\begin{equation}
\label{eq:new-positive-coercivity} \kappa\mathcal F s^2Q \ge c_0\mathcal F^{2/(m-q)}.
\end{equation}

\proofstep{Step 6: Removal of the remaining $u$-dependence.}

We estimate the gradient-error coefficient in terms of $\mathcal F$ by considering two ranges of
$P$. Choose $T\ge\max\{P_0,\sqrt2,1\}$, where $P_0$ is given by Lemma~\ref{lem:new-Q-positive}.
First suppose that
\[
P\ge T\varepsilon^{-2/\ell}.
\]
By \eqref{eq:new-chi-est},
\[
P\chi \le \frac{\varepsilon^{-2/\ell}}{P} \le T^{-1}.
\]
Moreover,
\[
\varepsilon P^\ell \ge T^\ell\varepsilon^{-1},
\]
and hence
\[
\chi = (1+\varepsilon P^\ell)^{-2/\ell} \le T^{-2}\le\frac12.
\]
Therefore
\[
\kappa\ge c_*>0.
\]
Combining this lower bound with \eqref{eq:new-logF-derivative-bounds} and the identities
preceding \eqref{eq:new-ratio-bounds}, we obtain
\[
|t| + \left| \frac{\mathcal F_{\rho\rho}}{\mathcal F_\rho}-1 \right| + \left| \frac{\mathcal F_{u\rho}}{\mathcal F_\rho} \right| \le C
\]
uniformly in this region. Since $P\ge T\ge P_0$ and $\chi\le1/2$, \eqref{eq:new-Q-strong} gives
\[
Q\ge c_*P^2.
\]
Hence the linear term in \eqref{eq:new-offcritical-div} satisfies
\[
C(1+P)s|E| \le CPs|E| \le \delta s^2Q+C|E|^2.
\]
Repeating the absorption argument of Step~4 with the same $\delta$ and using $\kappa\ge c_*>0$,
we obtain
\begin{equation}
\label{eq:new-large-P} \operatorname{div}(\A\nabla\mathcal F) \ge c_*\kappa\mathcal F s^2Q -
C\frac{|\nabla\mathcal F|^2}{\mathcal F}
\end{equation}
whenever
\[
P\ge T\varepsilon^{-2/\ell}.
\]

We now consider the complementary region
\[
P\le T\varepsilon^{-2/\ell} = T(u+c)^{8/\ell}.
\]
By the definition of $P$,
\[
s^{m-q} = \frac{\lambda e^u}{(m-1)P} \ge \frac{\lambda}{(m-1)T} e^u(u+c)^{-8/\ell}.
\]
Since $\phi\ge1$,
\[
\mathcal F = (u+c)^{-\beta}s^{m-q}\phi \ge \frac{\lambda}{(m-1)T} e^u (u+c)^{-\beta-8/\ell}.
\]
Taking logarithms gives
\[
\log\mathcal F \ge u -\left(\beta+\frac8\ell\right)\log(u+c) -C.
\]
Since logarithmic growth is dominated by linear growth, there exists $u_0>0$ such that for $u\ge
u_0$,
\[
\left(\beta+\frac8\ell\right)\log(u+c)+C \le\frac u2.
\]
Thus
\[
\log\mathcal F\ge\frac u2 \qquad(u\ge u_0).
\]
For $0\le u\le u_0$ the quantity $u+c$ is uniformly bounded. Combining the two cases gives
\begin{equation}
\label{eq:new-u-logF} u+c \le C\bigl(1+\log_+\mathcal F\bigr).
\end{equation}
Consequently,
\[
(u+c)^{20/\ell+2} \le C \bigl(1+\log_+\mathcal F\bigr)^{20/\ell+2}.
\]

Set
\[
J=\frac{20}{\ell}+2.
\]
Apply \eqref{eq:new-positive-coercivity} to \eqref{eq:new-large-P} in the first region, and to
\eqref{eq:new-coarse-div} together with \eqref{eq:new-u-logF} in the second. Decreasing $c_0>0$
if necessary and choosing $C_0>0$ large enough, we obtain
\[
\operatorname{div}(\A\nabla\mathcal F) \ge c_0\mathcal F^{2/(m-q)} - C_0 (1+\log_+\mathcal F)^J
\frac{|\nabla\mathcal F|^2}{\mathcal F}.
\]
This proves \eqref{eq:new-autonomous-divergence}.

\end{proof}

\subsection{A global vanishing argument}

\begin{lemma}[Global vanishing lemma]\label{lem:new-linear-KO}

Assume $\Ric_g\ge0$. Let $B$ be a symmetric, measurable tensor satisfying
\[
\nu I\le B\le\Lambda I, \qquad 0<\nu\le\Lambda<\infty.
\]
Suppose that
\[
w\in C(M)\cap W^{1,2}_{\mathrm{loc}}(M), \qquad w\ge0,
\]
satisfies
\[
\operatorname{div}(B\nabla w)\ge g(w)
\]
in the weak sense. Assume that $g:[0,\infty)\to[0,\infty)$ is continuous, $g(0)=0$, $g(t)>0$ for
$t>0$, and $g(t)/t$ is nondecreasing on $(0,\infty)$. If
\[
g(t)\ge c_*t(\log t)^4 \qquad(t\ge T_*)
\]
for some constants $c_*,T_*>0$, then $w\equiv0$.

\end{lemma}

\begin{proof}

The constants in the local estimates below depend only on $n,\nu,\Lambda$. Since $w$ is
continuous, it is bounded on compact sets; the compactly supported Sobolev test functions used
below are therefore admissible by approximation.

\proofstep{Step 1: A local mean value estimate.}

Since $g(w)\ge0$, we have
\[
-\int_M\langle B\nabla w,\nabla\varphi\rangle\,d\mu\ge0
\]
for every nonnegative, compactly supported $\varphi\in W^{1,2}_0(M)$. For $p\ge2$, $\delta>0$,
and a nonnegative $\eta\in C_0^\infty(M)$, take $\varphi=\eta^2(w+\delta)^{p-1}$. Ellipticity and
Young's inequality give
\[
\begin{aligned}
\nu(p-1)\int\eta^2(w+\delta)^{p-2}|\nabla w|^2\,d\mu &\le2\Lambda\int\eta(w+\delta)^{p-1}
|\nabla w||\nabla\eta|\,d\mu\\
&\le\frac{\nu(p-1)}2\int\eta^2(w+\delta)^{p-2}|\nabla w|^2\,d\mu\\
&\quad+\frac{2\Lambda^2}{\nu(p-1)} \int(w+\delta)^p|\nabla\eta|^2\,d\mu.
\end{aligned}
\]
Consequently,
\[
\int\eta^2(w+\delta)^{p-2}|\nabla w|^2\,d\mu
\le\frac{C}{(p-1)^2}\int(w+\delta)^p|\nabla\eta|^2\,d\mu.
\]
Expanding $\nabla(\eta(w+\delta)^{p/2})$ and letting $\delta\downarrow0$, we obtain
\begin{equation}
\label{eq:new-KO-caccioppoli} \int\left|\nabla(\eta w^{p/2})\right|^2\,d\mu \le Cp^2\int
w^p|\nabla\eta|^2\,d\mu.
\end{equation}

Fix $x\in M$ and $r>0$, and set $\mathcal V=\mu(B_r(x))$. Apply Lemma~\ref{lem:sobolev} with
scale $r/2$ and $K=0$. Its exponent is
\[
\chi=
\begin{cases}
n/(n-2),&n>2,\\
2,&n=2.
\end{cases}
\]
For $r/2\le\rho'<\rho\le r$, choose $\eta\in C_0^\infty(B_\rho(x))$ with $0\le\eta\le1$, $\eta=1$ on
$B_{\rho'}(x)$, and $|\nabla\eta|\le2/(\rho-\rho')$. Applying the Sobolev inequality, extended
by density to $\eta w^{p/2}\in W^{1,2}_0(B_r(x))$, and then \eqref{eq:new-KO-caccioppoli}, gives
\[
\begin{aligned}
\left(\frac1{\mathcal V}\int_{B_{\rho'}(x)}w^{p\chi}\,d\mu\right)^{1/\chi}
&\le\frac{Cr^2}{\mathcal V} \int\left|\nabla(\eta w^{p/2})\right|^2\,d\mu +\frac C{\mathcal V}\int_{B_\rho(x)}w^p\,d\mu\\
&\le\frac{Cp^2r^2}{(\rho-\rho')^2} \frac1{\mathcal V}\int_{B_\rho(x)}w^p\,d\mu.
\end{aligned}
\]
Thus
\begin{equation}
\label{eq:new-KO-moser-step}
\begin{aligned}
&\left(\frac1{\mathcal V}\int_{B_{\rho'}(x)}w^{p\chi}\,d\mu\right)^{1/(p\chi)}\\
&\qquad\le\left(\frac{Cp^2r^2}{(\rho-\rho')^2}\right)^{1/p} \left(\frac1{\mathcal V}\int_{B_\rho(x)}w^p\,d\mu\right)^{1/p}.
\end{aligned}
\end{equation}
Choose
\[
p_j=2\chi^j, \qquad \rho_j=\frac r2+\frac r{2^{j+1}}, \qquad j\ge0.
\]
Then $p_{j+1}=\chi p_j$, $\rho_0=r$, and $\rho_j-\rho_{j+1}=r/2^{j+2}$. Iterating
\eqref{eq:new-KO-moser-step} with these exponents and radii introduces at most the factor
\[
\prod_{j=0}^\infty(Cp_j^2 4^{j+2})^{1/p_j}\le C,
\]
which is finite because $\sum_{j\ge0}(1+j)\chi^{-j}<\infty$. Letting $j\to\infty$ yields
\[
\sup_{B_{r/2}(x)}w \le C\left(\frac1{\mu(B_r(x))}\int_{B_r(x)}w^2\,d\mu\right)^{1/2}.
\]
In particular,
\begin{equation}
\label{eq:new-linear-mean} w(x)^2\le\frac C{\mu(B_r(x))}\int_{B_r(x)}w^2\,d\mu.
\end{equation}

\proofstep{Step 2: A positive portion of the ball is a high-level set.}

Suppose that
\[
w(x)=M>0, \qquad \sup_{B_{2r}(x)}w\le2M.
\]
Let $C_{\mathrm{mv}}\ge2$ be a constant for which \eqref{eq:new-linear-mean} holds. For
$0<\tau<1$, define
\[
A=\{y\in B_r(x):w(y)\ge\tau M\}.
\]
The mean value estimate and the bound $w\le2M$ give
\[
\begin{aligned}
C_{\mathrm{mv}}^{-1}M^2\mu(B_r(x)) &\le\int_{B_r(x)}w^2\,d\mu\\
&\le\tau^2M^2\mu(B_r(x))+4M^2\mu(A).
\end{aligned}
\]
Choose $\tau=(2C_{\mathrm{mv}})^{-1/2}$ and $\sigma=(8C_{\mathrm{mv}})^{-1}$. Then
\begin{equation}
\label{eq:new-KO-level-measure} \mu\{y\in B_r(x):w(y)\ge\tau M\} \ge\sigma\mu(B_r(x)).
\end{equation}
Both $\tau$ and $\sigma$ depend only on $n,\nu,\Lambda$.

\proofstep{Step 3: The radius estimate.}

Under the assumptions of Step~2, choose $\eta\in C_0^\infty(B_{2r}(x))$ with $0\le\eta\le1$, $\eta=1$
on $B_r(x)$, and $|\nabla\eta|\le C/r$. Testing $\operatorname{div}(B\nabla w)\ge g(w)$ with
$\eta^2w$ gives
\[
\begin{aligned}
\int g(w)\eta^2w\,d\mu +\nu\int\eta^2|\nabla w|^2\,d\mu &\le2\Lambda\int\eta w|\nabla
w||\nabla\eta|\,d\mu\\
&\le\frac\nu2\int\eta^2|\nabla w|^2\,d\mu +\frac{2\Lambda^2}{\nu}\int w^2|\nabla\eta|^2\,d\mu.
\end{aligned}
\]
Therefore
\begin{equation}
\label{eq:new-KO-energy} \int g(w)\eta^2w\,d\mu +\frac\nu2\int\eta^2|\nabla w|^2\,d\mu \le\frac
C{r^2}\int_{B_{2r}(x)}w^2\,d\mu.
\end{equation}
The assumptions on $g(t)/t$ and $g\ge0$ imply that $g$ is nondecreasing. On the set $A$ from
Step~2, $g(w)w\ge\tau M g(\tau M)$. Using \eqref{eq:new-KO-level-measure}, we obtain
\[
\int g(w)\eta^2w\,d\mu \ge\tau\sigma M g(\tau M)\mu(B_r(x)).
\]
On the other hand, $w\le2M$ on $B_{2r}(x)$, and Bishop--Gromov comparison gives
$\mu(B_{2r}(x))\le2^n\mu(B_r(x))$. Hence \eqref{eq:new-KO-energy} implies
\[
\tau\sigma M g(\tau M)\mu(B_r(x)) \le\frac{CM^2}{r^2}\mu(B_r(x)).
\]
Cancelling the positive factors yields
\begin{equation}
\label{eq:new-KO-ball} r^2\frac{g(\tau M)}M\le C.
\end{equation}

\proofstep{Step 4: The doubling scale.}

For $M>0$, define
\[
R(M)=\left(\frac M{g(\tau M)}\right)^{1/2}.
\]
Since $g(t)/t$ is nondecreasing,
\[
R(M_2)\le R(M_1)\qquad(M_2\ge M_1>0).
\]
Fix $M>0$. For all sufficiently large $j$, the growth assumption on $g$ gives
\[
g(\tau2^jM)\ge c_*\tau2^jM\bigl(\log(\tau2^jM)\bigr)^4,
\]
and therefore
\[
R(2^jM) \le\frac1{\sqrt{c_*\tau}\,[j\log2+\log(\tau M)]^2} \le\frac{C_M}{(1+j)^2}.
\]
Thus the fourth power of the logarithm produces a summable tail:
\begin{equation}
\label{eq:new-KO-series} \sum_{j=0}^\infty R(2^jM)<\infty.
\end{equation}
The finitely many initial terms are finite because $g(t)>0$ for $t>0$.

\proofstep{Step 5: The doubling contradiction.}

Let $C_1$ be the constant in \eqref{eq:new-KO-ball}, and fix $K_0>\sqrt{C_1}$. Suppose that
$M_0:=w(x_0)>0$ at some $x_0\in M$. Set $r_0=K_0R(M_0)$. If $\sup_{B_{2r_0}(x_0)}w\le2M_0$, then
\eqref{eq:new-KO-ball} gives
\[
K_0^2=r_0^2\frac{g(\tau M_0)}{M_0}\le C_1,
\]
a contradiction. Hence there exists $x_1\in B_{2r_0}(x_0)$ with $M_1:=w(x_1)>2M_0$.

Repeating the argument gives a sequence satisfying
\[
M_j:=w(x_j), \qquad M_{j+1}>2M_j, \qquad d(x_j,x_{j+1})<2K_0R(M_j).
\]
By induction and the monotonicity of $R$,
\[
M_j\ge2^jM_0, \qquad d(x_j,x_{j+1})\le2K_0R(2^jM_0) \qquad(j\ge0).
\]
For $k>j$, it follows that
\[
d(x_j,x_k)\le2K_0\sum_{i=j}^{k-1}R(2^iM_0) \le2K_0\sum_{i=j}^\infty R(2^iM_0).
\]
By \eqref{eq:new-KO-series}, the last expression tends to zero as $j\to\infty$. Thus $(x_j)$ is
Cauchy and, by completeness, converges to a point $x_\infty\in M$. Continuity gives $w(x_j)\to
w(x_\infty)<\infty$, whereas $w(x_j)=M_j\ge2^jM_0\to\infty$. This contradiction proves
$w\equiv0$.

\end{proof}

\begin{proof}[Proof of Theorem~\ref{thm:liouville} for $m-1<q<m$]

The parameter $c$ remains fixed as in Lemma~\ref{lem:new-Q-positive}. Here $c_*$ and $C$ denote
positive constants depending only on $n,m,q,\lambda$, whose values may change from line to line.

\proofstep{Step 1: Removing the gradient-square error.}

Let $\mathcal F$ be the auxiliary function in Lemma~\ref{lem:new-divergence}, and set
$\nu=\min\{1,m-1\}$. Fix $b$, depending only on $n,m,q,\lambda$, such that
\begin{equation}
\label{chooseb} b>\frac{C_0}{\nu}.
\end{equation}
For $t>0$, define
\begin{equation}
\label{eq:new-Psi} \Psi'(t)=\exp\left(b\int_1^t \frac{(1+\log_+s)^J}{s}\,ds\right), \qquad
\Psi(t)=\int_0^t\Psi'(s)\,ds.
\end{equation}
Then
\begin{equation}
\label{Psi''} \Psi''(t)=b\frac{(1+\log_+t)^J}{t}\Psi'(t)>0.
\end{equation}
For $0<t\le1$, we have $\Psi'(t)=t^b$ and $\Psi(t)=t^{b+1}/(b+1)$. Thus $\Psi$ extends
continuously to $\Psi(0)=0$, is $C^2$ on $(0,\infty)$, and is strictly increasing and convex.
Since $\Psi'(t)\ge1$ for $t\ge1$, it maps $[0,\infty)$ onto $[0,\infty)$.

Using \eqref{eq:new-autonomous-divergence}, \eqref{Psi''}, and $\A\ge\nu I$, we obtain on
$\{\mathcal F>0\}$
\[
\begin{aligned}
\operatorname{div}(\A\nabla\Psi(\mathcal F))
&=\Psi'(\mathcal F)\operatorname{div}(\A\nabla\mathcal F)\\
&\quad+\Psi''(\mathcal F)\langle\A\nabla\mathcal F,\nabla\mathcal F\rangle\\
&\ge c_0\Psi'(\mathcal F)\mathcal F^{2/(m-q)}\\
&\quad+(\nu b-C_0)\Psi'(\mathcal F)(1+\log_+\mathcal F)^J \frac{|\nabla\mathcal F|^2}{\mathcal F}.
\end{aligned}
\]
By \eqref{chooseb}, the last term is nonnegative. Hence
\begin{equation}
\label{eq:new-Psi-div} \operatorname{div}(\A\nabla\Psi(\mathcal F)) \ge c_0\Psi'(\mathcal F)\mathcal F^{2/(m-q)} \qquad(\mathcal F>0).
\end{equation}

\proofstep{Step 2: Rewriting the source as a function of $W=\Psi(\mathcal F)$.}

Set $W=\Psi(\mathcal F)$. Since $\Psi'$ is increasing, $W\le\mathcal F\Psi'(\mathcal F)$, or
equivalently,
\begin{equation}
\label{Psi'} \Psi'(\mathcal F)\ge\frac W{\mathcal F}.
\end{equation}
It follows that
\[
c_0\Psi'(\mathcal F)\mathcal F^{2/(m-q)} \ge c_0W\mathcal F^{(q-m+2)/(m-q)}.
\]
If $0<\mathcal F\le1$, then $\mathcal F=((b+1)W)^{1/(b+1)}$. Therefore, for
$0<W\le1/(b+1)$,
\begin{equation}
\label{eq:new-small-W} c_0\Psi'(\mathcal F)\mathcal F^{2/(m-q)} \ge
c_*W^{1+\frac{q-m+2}{(m-q)(b+1)}}.
\end{equation}

For $\mathcal F\ge1$, the definition of $\Psi'$ gives
\[
\log\Psi'(\mathcal F) =\frac b{J+1}\left[(1+\log\mathcal F)^{J+1}-1\right].
\]
Consequently, \eqref{Psi'} implies
\[
\begin{aligned}
\log\Psi(\mathcal F) &\le\log\mathcal F +\frac b{J+1}\left[(1+\log\mathcal F)^{J+1}-1\right]\\
&\le C(1+\log\mathcal F)^{J+1}.
\end{aligned}
\]
For sufficiently large $\mathcal F$, we also have $\Psi(\mathcal F)\ge1$. In this range,
\[
(1+\log\Psi(\mathcal F))^4 \le C(1+\log\mathcal F)^{4(J+1)}.
\]
Since $(q-m+2)/(m-q)>0$,
\[
\frac{\mathcal F^{(q-m+2)/(m-q)}} {(1+\log\mathcal F)^{4(J+1)}}\longrightarrow\infty
\qquad(\mathcal F\to\infty).
\]
Thus, for all sufficiently large $W=\Psi(\mathcal F)$,
\begin{equation}
\label{eq:new-large-W} c_0\Psi'(\mathcal F)\mathcal F^{2/(m-q)} \ge c_*W(1+\log W)^4.
\end{equation}

The remaining values of $W$ lie in a compact subinterval of $(0,\infty)$. Since $\Psi$ is
strictly increasing and the source in \eqref{eq:new-Psi-div} is positive and continuous there,
\eqref{eq:new-small-W} and \eqref{eq:new-large-W} allow us to choose $c_1>0$ such that
\[
g(W)=c_1
\begin{cases}
W^{1+\frac{q-m+2}{(m-q)(b+1)}},&0\le W\le1,\\
W(1+\log W)^4,&W\ge1
\end{cases}
\]
satisfies
\[
c_0\Psi'(\mathcal F)\mathcal F^{2/(m-q)}\ge g(\Psi(\mathcal F)) \qquad(\mathcal F>0).
\]
The two branches of $g$ agree at $W=1$. Moreover, $g(0)=0$, $g(W)>0$ for $W>0$, and $g(W)/W$ is
nondecreasing on $(0,\infty)$. These properties and the large-$W$ growth show that $g$
satisfies all the assumptions of Lemma~\ref{lem:new-linear-KO}.
By \eqref{eq:new-Psi-div},
\begin{equation}
\label{eq:new-Psi-g} \operatorname{div}(\A\nabla\Psi(\mathcal F)) \ge g(\Psi(\mathcal F))
\qquad\text{on }\{\mathcal F>0\}.
\end{equation}

\proofstep{Step 3: Extension across the critical set.}

Fix $\delta>0$ and define
\[
W_\delta=(\Psi(\mathcal F)-\Psi(\delta))_+.
\]
Since $\mathcal F$ is continuous and vanishes on $\{|\nabla u|=0\}$,
\[
\{\mathcal F\ge\delta\}\subset\{|\nabla u|>0\}.
\]
For every compact $K\subset M$, the set $K\cap\{\mathcal F\ge\delta\}$ is compactly contained in
the noncritical set. It therefore has a neighborhood on which $u,\mathcal F$, and $\A$ are
smooth. On this neighborhood, $W_\delta$ is the positive part of a $C^2$ function; near every
point with $\mathcal F<\delta$, it is identically zero. Hence $W_\delta$ is locally Lipschitz
and, in particular,
\begin{equation}
\label{Wdelta} W_\delta\in C(M)\cap W^{1,2}_{\mathrm{loc}}(M).
\end{equation}
Its weak gradient is
\[
\nabla W_\delta =\mathbf1_{\{\mathcal F>\delta\}}\Psi'(\mathcal F)\nabla\mathcal F
\qquad\text{a.e.},
\]
with the right-hand side taken to be zero outside $\{\mathcal F>\delta\}$. Extend $\A$ by
setting $\A=I$ on the critical set. The resulting tensor is measurable and symmetric, and
satisfies
\[
\nu I\le\A\le\max\{1,m-1\}I \qquad\text{on }M.
\]

Choose smooth nondecreasing functions $\theta_\eta:\mathbb R\to[0,1]$ with
\[
\theta_\eta(t)=0\quad(t\le0), \qquad\theta_\eta(t)=1\quad(t\ge\eta), \qquad\theta_\eta'(t)\ge0.
\]
For a nonnegative $\zeta\in C_0^\infty(M)$, the function $\zeta\theta_\eta(W-\Psi(\delta))$ has
compact support in the noncritical set. It is an admissible test function in
\eqref{eq:new-Psi-g} by smooth approximation within that set, where $W=\Psi(\mathcal F)$ is
$C^2$. Expanding its derivative gives
\[
\begin{aligned}
&-\int\theta_\eta(W-\Psi(\delta)) \langle\A\nabla W,\nabla\zeta\rangle\,d\mu\\
&\quad-\int\zeta\theta_\eta'(W-\Psi(\delta)) \langle\A\nabla W,\nabla W\rangle\,d\mu\\
&\qquad\ge\int g(W)\zeta\theta_\eta(W-\Psi(\delta))\,d\mu.
\end{aligned}
\]
Since $\theta_\eta'\ge0$ and $\A$ is positive definite, the second term on the left is
nonpositive. Dropping it yields
\[
-\int\theta_\eta(W-\Psi(\delta)) \langle\A\nabla W,\nabla\zeta\rangle\,d\mu \ge\int
g(W)\zeta\theta_\eta(W-\Psi(\delta))\,d\mu.
\]
All integrands are supported in a fixed compact subset of the noncritical set. Letting
$\eta\downarrow0$ and using dominated convergence together with the gradient formula for
$W_\delta$, we obtain
\[
-\int\langle\A\nabla W_\delta,\nabla\zeta\rangle\,d\mu \ge\int_{\{\mathcal F>\delta\}}g(W)\zeta\,d\mu.
\]
On $\{\mathcal F>\delta\}$, we have $0\le W_\delta\le W$, so $g(W_\delta)\le g(W)$ by
monotonicity. On its complement, $g(W_\delta)=g(0)=0$. Therefore
\[
-\int\langle\A\nabla W_\delta,\nabla\zeta\rangle\,d\mu \ge\int g(W_\delta)\zeta\,d\mu.
\]
Equivalently,
\begin{equation}
\label{eq:new-Wdelta} \operatorname{div}(\A\nabla W_\delta)\ge g(W_\delta) \qquad\text{weakly on
}M.
\end{equation}

By \eqref{Wdelta}, the ellipticity of $\A$, and the properties of $g$,
Lemma~\ref{lem:new-linear-KO} applies to \eqref{eq:new-Wdelta} and gives $W_\delta\equiv0$. Thus
$\Psi(\mathcal F)\le\Psi(\delta)$ and, since $\Psi$ is strictly increasing, $\mathcal F\le\delta$. Letting $\delta\downarrow0$ yields $\mathcal F\equiv0$. By
\eqref{eq:new-auxiliary}, $\mathcal F>0$ wherever $|\nabla u|>0$, so $|\nabla u|\equiv0$. Since
$M$ is connected, $u$ is constant.

\end{proof}
\section{The cases \texorpdfstring{$q\le0$, $q=m$, and $0<q\le m-1$}{q <= 0, q = m, and 0 < q <= m-1}}\label{sec:endpoints}

\begin{proof}[Proof of Theorem~\ref{thm:liouville} for $q\le0$, $q=m$, and $0<q\le m-1$]

Let $u>0$ be a $C^1$ weak solution of~\eqref{eq:main}.
Fix $o\in M$ and choose $\eta_R\in C_0^\infty(M)$ such that
\[
 0\le\eta_R\le1,\qquad
 \eta_R=1\ \text{on }B(o,R),\qquad
 \supp\eta_R\subset B(o,2R),\qquad
 |\nabla\eta_R|\le C/R.
\]
Choose an integer $N>m$.
The composite test functions used below are admissible by smooth approximation
both uniformly and in $W^{1,m}$, since the right-hand side of~\eqref{eq:main} is locally integrable.

\proofstep{Case 1: $q\le0$.}

When $q<0$, the local integrability of $e^u|\nabla u|^q$ implies that
$|\nabla u|>0$ almost everywhere.
For all $q\le0$, testing~\eqref{eq:main} with $e^{-\frac{m}{m-q}u}\eta_R^N$ yields
\[
 \begin{aligned}
 &\lambda\int e^{-\frac{q}{m-q}u}|\nabla u|^q\eta_R^N\dmu
 +\frac{m}{m-q}\int e^{-\frac{m}{m-q}u}|\nabla u|^m\eta_R^N\dmu\\
 &\quad=N\int e^{-\frac{m}{m-q}u}|\nabla u|^{m-2}
 \eta_R^{N-1}\langle\nabla u,\nabla\eta_R\rangle\dmu\\
 &\quad\le N\int e^{-\frac{m}{m-q}u}|\nabla u|^{m-1}
 \eta_R^{N-1}|\nabla\eta_R|\dmu.
 \end{aligned}
\]
Applying Young's inequality, using $N\ge m$ and $0\le\eta_R\le1$, and absorbing half of the gradient term gives
\[
 \begin{aligned}
 &\int\left[
 \lambda\left(e^{-\frac{u}{m-q}}|\nabla u|\right)^q
 +\frac{m}{2(m-q)}\left(e^{-\frac{u}{m-q}}|\nabla u|\right)^m
 \right]\eta_R^N\dmu\\
 &\qquad\le C\int e^{-\frac{m}{m-q}u}|\nabla\eta_R|^m\dmu
 \le CR^{-m}\mu(B(o,2R)).
 \end{aligned}
\]
For $q<0$, set
\[
c:=\min_{t>0}\left\{\lambda t^q+\frac{m}{2(m-q)}t^m\right\}>0.
\]
The expression is positive and tends to $+\infty$ as $t\downarrow0$ and $t\to\infty$, so its minimum is positive.
For $q=0$, the first term in the bracket equals $\lambda$, so we take $c=\lambda$.
In both cases, $c=c(m,q,\lambda)>0$ is independent of $R$.
Since $\eta_R=1$ on $B(o,R)$, Bishop--Gromov comparison gives
\[
c\,\mu(B(o,R))
\le C R^{-m}\mu(B(o,2R))
\le C2^nR^{-m}\mu(B(o,R)).
\]
Hence
\[
0<c\le C2^nR^{-m}\to0
\qquad (R\to\infty),
\]
a contradiction.

\proofstep{Case 2: $q=m$.}

Set
\[
v=\int_0^u\exp\!\left(\frac{\lambda e^t}{m-1}\right)\dd t.
\]
Then
\[
|\nabla v|^{m-2}\nabla v
=e^{\lambda e^u}|\nabla u|^{m-2}\nabla u,
\]
and, in the weak sense,
\[
\Delta_{m,g}v
=e^{\lambda e^u}
\left(\Delta_{m,g}u+\lambda e^u|\nabla u|^m\right)=0.
\]
Thus $v$ is a positive $m$-harmonic function. By the positive
$m$-harmonic Liouville theorem~\cite{WZ11}, $v$ is constant.
The defining change of variables is strictly increasing, so $u$ is constant as well.

\proofstep{Case 3: $0<q\le m-1$.}

Set $v=\log u$.
For $0\le\zeta\in C_0^\infty(M)$, use $u^{1-m}\zeta$ in the original equation.
This test is admissible because $u$ has a positive lower bound on $\supp\zeta$.
The identity $\nabla u=u\nabla v$ then gives
\begin{equation}\label{eq:low-log}
 -\Delta_{m,g}v
 =(m-1)|\nabla v|^m
 +\lambda e^u u^{q-m+1}|\nabla v|^q.
\end{equation}
Since $q\le m-1$, we have
\[
c_0:=\inf_{t>0}\lambda e^t t^{\,q-m+1}>0.
\]
Thus~\eqref{eq:low-log} gives
\[
-\Delta_{m,g}v
\ge
(m-1)|\nabla v|^m+c_0|\nabla v|^q.
\]

For $|\nabla v|\ge1$, the first power bounds $|\nabla v|^{m-1}$ from below;
for $|\nabla v|\le1$, the second power does so because $q\le m-1$.
Consequently, with
\[
c=\min\{m-1,c_0\}>0,
\]
we obtain
\begin{equation}\label{eq:low-pair}
-\Delta_{m,g}v\ge(m-1)|\nabla v|^m,
\qquad
-\Delta_{m,g}v\ge c|\nabla v|^{m-1}.
\end{equation}

Testing the first inequality with $\eta_R^N$ gives
\[
 (m-1)\int\eta_R^N|\nabla v|^m\dmu
 \le N\int\eta_R^{N-1}|\nabla v|^{m-1}|\nabla\eta_R|\dmu.
\]
Since $N\ge m$, Young's inequality and the Bishop--Gromov volume upper bound imply
\[
 \int_{B(o,R)}|\nabla v|^m\dmu
 \le CR^{-m}\mu(B(o,2R))\le CR^{n-m}.
\]
By H\"older's inequality,
\begin{equation}\label{eq:low-growth}
 \begin{split}
 \int_{B(o,R)}|\nabla v|^{m-1}\dmu
 &\le\left(\int_{B(o,R)}|\nabla v|^m\dmu\right)^{(m-1)/m}
       \mu(B(o,R))^{1/m}\\
 &\le CR^{n-m+1}.
 \end{split}
\end{equation}
Testing the second inequality in~\eqref{eq:low-pair} with $\eta_R$, we obtain
\[
 \int_{B(o,R)}|\nabla v|^{m-1}\dmu
 \le\frac C R\int_{B(o,2R)}|\nabla v|^{m-1}\dmu.
\]
Here $C$ may depend on $\lambda$ through $c$, but remains independent of $R$.
Apply this estimate successively at the radii $R,2R,\ldots,2^{k-1}R$, where $k$ is a fixed positive integer.
Using~\eqref{eq:low-growth} at radius $2^kR$ yields
\[
 \begin{split}
 \int_{B(o,R)}|\nabla v|^{m-1}\dmu
 &\le C^k2^{-k(k-1)/2}R^{-k}
       \int_{B(o,2^kR)}|\nabla v|^{m-1}\dmu\\
 &\le C_{k}\,R^{n-m+1-k}.
 \end{split}
\]
Choose $k>n-m+1$ and let $R\to\infty$.
Every fixed ball $B(o,r)$ is contained in $B(o,R)$ for all sufficiently large $R$,
so the integral over $B(o,r)$ vanishes.
Continuity gives $\nabla v=0$ everywhere.
Therefore $u$ is constant.
\end{proof}


\begin{thebibliography}{99}

\bibitem{BGV14}
M.-F.~Bidaut-V\'eron, M.~Garc\'ia-Huidobro and L.~V\'eron.
Local and global properties of solutions of quasilinear Hamilton--Jacobi equations.
\emph{J. Funct. Anal.} \textbf{267} (2014), 3294--3331.
\href{https://doi.org/10.1016/j.jfa.2014.07.003}{\nolinkurl{doi:10.1016/j.jfa.2014.07.003}}.

\bibitem{BGV19}
M.-F.~Bidaut-V\'eron, M.~Garc\'ia-Huidobro and L.~V\'eron.
Estimates of solutions of elliptic equations with a source reaction term involving the product of the function and its gradient.
\emph{Duke Math. J.} \textbf{168} (2019), 1487--1537.
\href{https://doi.org/10.1215/00127094-2018-0067}{\nolinkurl{doi:10.1215/00127094-2018-0067}}.

\bibitem{BV21}
M.-F.~Bidaut-V\'eron.
Liouville results and asymptotics of solutions of a quasilinear elliptic equation with supercritical source gradient term.
\emph{Adv. Nonlinear Stud.} \textbf{21} (2021), 57--76.
\href{https://doi.org/10.1515/ans-2020-2109}{\nolinkurl{doi:10.1515/ans-2020-2109}}.

\bibitem{BGV26exp}
M.-F.~Bidaut-V\'eron, M.~Garc\'ia-Huidobro and L.~V\'eron.
Local and global properties of solutions of an elliptic equation involving exponential and gradient reaction.
Preprint, \href{https://arxiv.org/abs/2606.06206}{arXiv:2606.06206} (2026).

\bibitem{BM}
A.~Bj\"orn and N.~Marola.
Moser iteration for (quasi)minimizers on metric spaces.
\emph{Manuscripta Math.} \textbf{121} (2006), 339--366.
\href{https://doi.org/10.1007/s00229-006-0040-8}{\nolinkurl{doi:10.1007/s00229-006-0040-8}}.

\bibitem{Buser}
P.~Buser.
A note on the isoperimetric constant.
\emph{Ann. Sci. \'Ecole Norm. Sup.} (4) \textbf{15} (1982), 213--230.
\href{https://doi.org/10.24033/asens.1426}{\nolinkurl{doi:10.24033/asens.1426}}.

\bibitem{CHZ22}
C.~Chang, B.~Hu and Z.~Zhang.
Liouville-type theorems and existence of solutions for quasilinear elliptic equations with nonlinear gradient terms.
\emph{Nonlinear Anal.} \textbf{220} (2022), Paper No.~112873.
\href{https://doi.org/10.1016/j.na.2022.112873}{\nolinkurl{doi:10.1016/j.na.2022.112873}}.

\bibitem{DiBenedetto83}
E.~DiBenedetto.
$C^{1+\alpha}$ local regularity of weak solutions of degenerate elliptic equations.
\emph{Nonlinear Anal.} \textbf{7} (1983), no.~8, 827--850.
\href{https://doi.org/10.1016/0362-546X(83)90061-5}{\nolinkurl{doi:10.1016/0362-546X(83)90061-5}}.

\bibitem{DSWZ25}
J.~Dou, B.~Shi, T.~Wu and H.~Zhu.
Classification of positive solutions to a class of Laplace equations with a gradient term.
\emph{J. Eur. Math. Soc. (JEMS)}, to appear.
Preprint, \href{https://arxiv.org/abs/2511.20205}{arXiv:2511.20205} (2025).

\bibitem{Fil09}
R.~Filippucci.
Nonexistence of positive weak solutions of elliptic inequalities.
\emph{Nonlinear Anal.} \textbf{70} (2009), 2903--2916.
\href{https://doi.org/10.1016/j.na.2008.12.018}{\nolinkurl{doi:10.1016/j.na.2008.12.018}}.

\bibitem{FPS20}
R.~Filippucci, P.~Pucci and P.~Souplet.
A Liouville-type theorem for an elliptic equation with superquadratic growth in the gradient.
\emph{Adv. Nonlinear Stud.} \textbf{20} (2020), 245--251.
\href{https://doi.org/10.1515/ans-2019-2070}{\nolinkurl{doi:10.1515/ans-2019-2070}}.

\bibitem{GT}
D.~Gilbarg and N.~S.~Trudinger.
\emph{Elliptic Partial Differential Equations of Second Order}.
2nd ed., Classics in Mathematics, Springer, Berlin, 2001.
\href{https://doi.org/10.1007/978-3-642-61798-0}{\nolinkurl{doi:10.1007/978-3-642-61798-0}}.

\bibitem{Goffi23}
A.~Goffi.
H\"older regularity and Liouville properties for nonlinear elliptic inequalities with power-growth gradient terms.
\emph{Proc. Roy. Soc. Edinburgh Sect. A} \textbf{153} (2023), 1833--1857.
\href{https://doi.org/10.1017/prm.2022.72}{\nolinkurl{doi:10.1017/prm.2022.72}}.

\bibitem{HanHeWang25}
D.~Han, J.~He and Y.~Wang.
Universal gradient estimates for quasilinear Hamilton--Jacobi type equation on manifolds and Liouville theorems.
\emph{J. Differential Equations} \textbf{449} (2025), Paper No.~113724.
\href{https://doi.org/10.1016/j.jde.2025.113724}{\nolinkurl{doi:10.1016/j.jde.2025.113724}}.

\bibitem{HeHuWang26}
J.~He, J.~Hu and Y.~Wang.
Nash--Moser iteration approach to the logarithmic gradient estimates and Liouville properties of quasilinear elliptic equations on manifolds.
\emph{Math. Z.} \textbf{313} (2026), Paper No.~6.
\href{https://doi.org/10.1007/s00209-026-04030-7}{\nolinkurl{doi:10.1007/s00209-026-04030-7}}.

\bibitem{LiangZhang26}
W.~Liang and Z.~Zhang.
Liouville-type theorems and universal estimates for semilinear elliptic equations involving the product of the function and its gradient.
Preprint, \href{https://arxiv.org/abs/2609.17149}{arXiv:2609.17149} (2026).

\bibitem{Lions85}
P.-L.~Lions.
Quelques remarques sur les probl\`emes elliptiques quasilin\'eaires du second ordre.
\emph{J. Anal. Math.} \textbf{45} (1985), 234--254.
\href{https://doi.org/10.1007/BF02792551}{\nolinkurl{doi:10.1007/BF02792551}}.

\bibitem{Lu26}
Z.~Lu.
Global and local properties of solutions of elliptic equations with a nonlinear term involving the product of the function and its gradient.
Preprint, \href{https://arxiv.org/abs/2602.19999}{arXiv:2602.19999} (2026).

\bibitem{LZ}
Z.~Lu and L.~Zhu.
Liouville theorems and gradient estimates for a class of quasilinear elliptic equations with product gradient structures.
\emph{J. Funct. Anal.} \textbf{291} (2026), Paper No.~111620.
\href{https://doi.org/10.1016/j.jfa.2026.111620}{\nolinkurl{doi:10.1016/j.jfa.2026.111620}}.

\bibitem{MaWu26}
X.-N.~Ma and W.~Wu.
Liouville theorem for elliptic equations with a source reaction term involving the product of the function and its gradient in $\mathbb R^n$.
\emph{Bull. Sci. Math.} \textbf{206} (2026), Paper No.~103747.
\href{https://doi.org/10.1016/j.bulsci.2025.103747}{\nolinkurl{doi:10.1016/j.bulsci.2025.103747}}.

\bibitem{MP99}
E.~Mitidieri and S.~I.~Pokhozhaev.
Nonexistence of positive solutions for quasilinear elliptic problems on $\mathbb R^N$.
\emph{Proc. Steklov Inst. Math.} \textbf{227} (1999), 186--216.
\url{https://www.mathnet.ru/eng/tm557}.

\bibitem{SC}
L.~Saloff-Coste.
Uniformly elliptic operators on Riemannian manifolds.
\emph{J. Differential Geom.} \textbf{36} (1992), 417--450.
\href{https://doi.org/10.4310/jdg/1214448748}{\nolinkurl{doi:10.4310/jdg/1214448748}}.

\bibitem{Sun15}
Y.~Sun.
On the uniqueness of nonnegative solutions of differential inequalities with gradient terms on Riemannian manifolds.
\emph{Commun. Pure Appl. Anal.} \textbf{14} (2015), 1743--1757.
\href{https://doi.org/10.3934/cpaa.2015.14.1743}{\nolinkurl{doi:10.3934/cpaa.2015.14.1743}}.

\bibitem{SX21}
Y.~Sun and F.~Xu.
Liouville's theorems to quasilinear differential inequalities involving gradient nonlinearity term on manifolds.
Preprint, \href{https://arxiv.org/abs/2102.02073}{arXiv:2102.02073} (2021).

\bibitem{SXX22}
Y.~Sun, J.~Xiao and F.~Xu.
A sharp Liouville principle for $\Delta_m u+u^p|\nabla u|^q\le0$ on geodesically complete noncompact Riemannian manifolds.
\emph{Math. Ann.} \textbf{384} (2022), 1309--1341.
\href{https://doi.org/10.1007/s00208-021-02311-6}{\nolinkurl{doi:10.1007/s00208-021-02311-6}}.

\bibitem{Tolksdorf84}
P.~Tolksdorf.
Regularity for a more general class of quasilinear elliptic equations.
\emph{J. Differential Equations} \textbf{51} (1984), no.~1, 126--150.
\href{https://doi.org/10.1016/0022-0396(84)90105-0}{\nolinkurl{doi:10.1016/0022-0396(84)90105-0}}.

\bibitem{WZ11}
X.~Wang and L.~Zhang.
Local gradient estimate for $p$-harmonic functions on Riemannian manifolds.
\emph{Comm. Anal. Geom.} \textbf{19} (2011), 759--771.
\href{https://doi.org/10.4310/CAG.2011.v19.n4.a4}{\nolinkurl{doi:10.4310/CAG.2011.v19.n4.a4}}.

\bibitem{WX16}
Y.~Wang and J.~Xiao.
A constructive approach to positive solutions of $\Delta_p u+f(u,\nabla u)\le0$ on Riemannian manifolds.
\emph{Ann. Inst. H. Poincar\'e C Anal. Non Lin\'eaire} \textbf{33} (2016), 1497--1507.
\href{https://doi.org/10.1016/j.anihpc.2015.06.003}{\nolinkurl{doi:10.1016/j.anihpc.2015.06.003}}.

\bibitem{WYZ26}
T.~Wu, J.~Yan and H.~Zhu.
The $m$-Laplace equation with a gradient term I: Liouville theorem in the second subcritical case.
Preprint, \href{https://arxiv.org/abs/2609.30006}{arXiv:2609.30006} (2026).

\bibitem{WYZ26II}
T.~Wu, J.~Yan and H.~Zhu.
The $m$-Laplace equation with a gradient term II: classification in the second critical case.
Preprint, \href{https://arxiv.org/abs/2609.33583}{arXiv:2609.33583} (2026).

\end{thebibliography}
\end{document}